\documentclass[reqno] {amsart}
\usepackage{amsmath}
\usepackage{amscd}
\usepackage{amsthm}
\usepackage{amssymb}
\usepackage{latexsym}
\usepackage{enumerate}

\newtheorem{thm}{Theorem}[section]
\newtheorem{prop}{Proposition}[section]
\newtheorem{lem}{Lemma}[section]
\newtheorem{cor}{Corollary}[section]

\allowdisplaybreaks

\title{Uniqueness of solutions, radiation conditions, and complexity of the metric at infinity}

\author{Hironori Kumura}

\email{smhkumu@ipc.shizuoka.ac.jp} 
\address{Department of Mathematics, Shizuoka University, 
Shizuoka 422-8529, Japan}

\date{}

\begin{document}

\maketitle

\begin{abstract}
The purpose of this paper is to prove the uniqueness theorem of solutions of eigenvalue equations on one end of Riemannian manifolds for drift Laplacians, including the standard Laplacian as a special case; we shall impose ``a sort of radiation condition'' at infinity on solutions. 
We shall also provide several Riemannian manifolds whose Laplacians satisfy the absence of embedded eigenvalues and besides the absolutely continuity, although growth orders of their metrics on ends are very complicated. 
\end{abstract}
%%
%%%%%%    INTRODUCTION    %%%%%%
%%
\section{Introduction}

The Laplace-Beltrami operator $\Delta_g$ on a noncompact complete Riemannian manifold $(M, g)$ is essentially self-adjoint on $C^{\infty}_0(M)$; the relationship between the spectral structure of its self-adjoint extension to $L^2(M, v_g)$ and the geometry of $(M, g)$ has been studied by several authors from various points of view. 
For example, the absence of eigenvalues was studied in [2-7, 9-11, 14, 15, 18, 20] and so on. 
This paper will treat the case where $(M, g)$ has specific types of end $E$, and show the uniqueness of solutions $f$ of eigenvalue equations $\Delta_g f + \langle \nabla w, \nabla f \rangle + \alpha f = 0$ on $E$ for drift Laplacians $\Delta_g + \nabla w$, imposing ``a kind of radiation condition'' at infinity; here, $w$ is a $C^{\infty}$-function on $E$ and $\alpha >0$ is a constant. 

We shall state our results precisely. 
Let $(M, g)$ be a noncompact connected complete Riemannian manifold and $U$ be an open subset of $M$. 
We shall say that $E := M-U$ is an {\it end with radial coordinates} if and only if the boundary $\partial E$ is $C^{\infty}$, compact, and connected, and the outward normal exponential map $\exp_{\partial E}^{\perp} : N^{+}(\partial E) \to E$ induces a diffeomorphism, where $N^{+}(\partial E) := \{ v\in T(\partial E) \mid v {\rm ~is~outward~normal~to~}\partial E \}$; note that $U$ is not necessarily relatively compact. 
We shall set $r:={\rm dist}(\partial E,*)$ on $E$. 
In the sequel, the following notations will be used:
\begin{align*}
& E(s, t) := \{ x \in E \mid s < r(x) < t \} \quad \mathrm{for}\ 0 \le s < t ; \\
& E(s, \infty) := \{x \in E \mid s<r(x)\} \quad \mathrm{for}\ 0 \le s < \infty; \\ 
& S(t) := \{x \in E \mid r(x)=t\} \quad \mathrm{for}\ 0 \le t< \infty;\\
& \widetilde{g} := g - dr \otimes dr.
\end{align*}
We denote the Riemannian measure of $(M,g)$ by $v_g$, and the measure on each $S(t)$ induced from $g$ simply by $A$ for $t\ge 0$. 
Let $w$ be a $C^{\infty}$-function on $M$. 
Our concern is to study a drift Laplacian $\Delta_g + \nabla w$: this operator is associated with the Dirichlet form
\begin{align*}
 \int_M \langle \nabla u, \nabla v \rangle e^w dv_g 
 \quad {\rm for}\ u, v \in H^1(M, e^w v_g),
\end{align*}
where $\nabla u$ stands for the gradient of $u$, and $\langle~, ~\rangle =g$. 
Let $\nabla dr$ denote the covariant derivative of $1$-form $dr$, that is, the Hessian of $r$. 
The main theorem of this paper is the following: 
%%
%%%%%     Theorem 1.1     %%%%%
%%
\begin{thm}
Let $(M,g)$ be a noncompact Riemannian manifold, $E$ be an end with radial coordinates of $(M,g)$, and $w$ be a $C^{\infty}$-function on $M$. 
Let $r$ denote ${\rm dist}\,(\partial E,*)$ on $E$. 
Assume that there exist constants $\widetilde{\alpha}_1 > 0$, $A_1 > 0$, $B_1 > 0$, $r_0 \ge 0$, $b \in \mathbb{R}$, and $c \in \mathbb{R}$ such that 
\begin{align}
 & \nabla dr \ge \frac{\widetilde{\alpha}_1}{r} \, \widetilde{g} 
   \quad {\rm on}~E(r_0,\infty); \tag{$*_1$} \\
 & -A_1 \le r \Big( \Delta_g r + \frac{\partial w}{\partial r} - c \Big) - b \le B_1 \quad {\rm on}~E(r_0,\infty). \tag{$*_2$}
\end{align}
Let $\alpha>0$ and $\gamma >0 $ be constants, and assume that $f$ is a solution of
\begin{align*}
 \Delta_g f + \langle \nabla w, \nabla f \rangle + \alpha f = 0 
 \quad {\rm on}\ \ E, 
\end{align*}
satisfying 
\begin{align}
 \liminf_{t\to \infty}~t^{\gamma} \int_{S(t)}
 \bigg\{ \left( \frac{\partial f}{\partial r} \right) ^2 
 + f^2 \bigg\} e^w dA = 0. \tag{$*_3$}
\end{align}
Let $\varepsilon_0$ be the constant defined by
\begin{align}
 \varepsilon_0
 := \min \bigg\{ \frac{2\gamma + A_1 - B_1}{2}, \ 2 \widetilde{\alpha }_1 - B_1 \bigg\}, \tag{$*_4$}
\end{align}
and assume that 
\begin{align}
& 2 \min \{ \widetilde{\alpha}_1, \gamma \} > A_1 + B_1~; \tag{$*_5$} \\ 
& \alpha > \frac{c^2}{4} \bigg\{ 1 
  + \frac{(A_1 + B_1)^2}{(2\gamma - \varepsilon_0 - B_1)(\varepsilon_0 - A_1)} \bigg\}. \tag{$*_6$}
\end{align}
Then, we have $f\equiv 0$ on $E$.
\end{thm}

Note that $\Delta_g + \frac{\partial w}{\partial r}$ expresses the growth order of the measure $e^w v_g$, and hence, $(*_2)$ implies that it converges to a constant $c \in \mathbb{R}$ at infinity. 
Note that, we do not assume that the constant $c$ is nonnegative; even if $c$ is negative, the conclusion of Theorem $1.1$ holds good because of the geometrical expansion condition $(*_1)$. 
Note also that we do not assume that $f\in L^2(E, e^w v_g)$ in Theorem $1.1$; indeed, Theorem $1.1$ with {\it small} $0<\gamma \ll 1$ is required to prove the limiting absorption principle for $-\Delta_g$ in author's paper \cite{K4} ; for details, see Section $6$ below. 

As for the technical constant $(*_4)$, note that $(*_5)$ implies that $A_1 < \varepsilon_0 \le \frac{2\gamma + A_1 - B_1}{2}$. 
Note that, if necessary, by replacing $\gamma > 0$ with smaller one, we may assume that $\widetilde{\alpha }_1 \ge \gamma$. 
Then, since $(*_5)$ implies $\varepsilon_0 = \frac{2\gamma + A_1 - B_1}{2}$, $(*_4)$ and $(*_6)$ are reduced to the following simpler form:
\begin{align}
 \alpha > \frac{c^2}{4} \left\{ 1  
  + \bigg( \frac{2 (A_1 + B_1)}{2\gamma - A_1 - B_1} \bigg)^2 \right\} 
  \tag{$*_7$}. 
\end{align}
In view of $(*_2)$ and $(*_6)$ (or $(*_7)$), we can see why ``small perturbation $\frac{\varepsilon}{r}$'' of $\Delta_g r + \frac{\partial w}{\partial r} - \frac{b}{r}$ is allowed for the absence of eigenvalues in case $c = 0$, which was first observed in the paper \cite{K3}.  

A drift Laplacian $-\Delta_g - \nabla w$ defined on $C_0^{\infty}(M)$ is essentially self-adjoint on $L^2(M, e^w v_g)$, and we shall denote its self-adjoint extension by the same symbol for simplicity. 
Then, note that $(*_2)$ implies that $\sigma_{\rm ess}\left(- \Delta_g - \nabla w \right) \supseteq [\frac{c^2}{4}, \infty)$, where $\sigma_{\rm ess} (-\Delta_g - \nabla w )$ stands for the essential spectrum of $-\Delta_g - \nabla w$ on $L^2(M, e^w v_g)$ (see, for example, \cite{K1}). 

By putting $\gamma=1$ in Theorem $1.1$, we obtain the following:
%%
%%%%%     Corollary 1.1     %%%%%
%%
\begin{cor}
Let $(M, g)$ be a noncompact connected complete Riemannian manifold, $E$ be an end with radial coordinates of $(M, g)$, and $w$ be a $C^{\infty}$-function on $M$. 
Let $r$ denote ${\rm dist}\,(\partial E,*)$ on $E$. 
Assume that there exist constants $\widetilde{\alpha}_1 > 0$, $A_1 > 0$, $B_1 > 0$, $r_0 \ge 0$, $b \in \mathbb{R}$, and $c \in \mathbb{R}$ such that 
\begin{align*}
  \nabla dr \ge \frac{\widetilde{\alpha}_1}{r} \,\widetilde{g} 
   \quad {\rm on}~E(r_0,\infty); \ 
  -A_1 \le r \Big(\Delta_g r + \frac{\partial w}{\partial r} - c \Big) -b 
  \le B_1 \quad {\rm on}~E(r_0,\infty). 
\end{align*}
Assume that $2 \min\{ \widetilde{\alpha}_1, 1 \} > A_1 + B_1$, and set 
\begin{align*}
 \varepsilon_1 
 := 
 \min \bigg\{ \frac{2 + A_1 - B_1}{2}, \ 2 \widetilde{\alpha }_1 - B_1 \bigg\}.
\end{align*}
Then, $\sigma_{\rm ess}(-\Delta_g - \nabla w) \supseteq [\frac{c^2}{4}, \infty)$ and  
\begin{align*}
 \sigma_{\rm pp}(-\Delta_g -\nabla w) \cap \left( \frac{c^2}{4} \bigg\{ 1 
 + \frac{(A_1 + B_1)^2}{(2 - \varepsilon_1 - B_1)(\varepsilon_1 - A_1)} \bigg\}, ~ \infty \right) = \emptyset ,
\end{align*}
where $\sigma_{\rm pp}(-\Delta_g -\nabla w)$ stands for the set of all eigenvalues of $-\Delta_g -\nabla w$ on $L^2(M, e^w v_g)$. 
\end{cor}

In case $\widetilde{\alpha}_1 \ge 1$, the condition ``~$2 \min\{ \widetilde{\alpha}_1, 1 \} > A_1 + B_1$'' implies $\varepsilon_1 = \frac{2 + A_1 - B_1}{2}$, and hence, the assertion in Corollary 1.1 is reduced to the following simpler one: 
\begin{align}
 \sigma_{\rm pp}(-\Delta_g -\nabla w) \cap 
 \left( \frac{c^2}{4} \bigg\{ 1 + \bigg( \frac{2 (A_1 + B_1)}{2 - A_1 - B_1} \bigg)^2 \bigg\} , ~ \infty \right) = \emptyset \tag{$*_8$}. 
\end{align}

In case $c=1$, $b=0$, and $w \equiv 0$, it seems to be interesting to compare Corollary $1.1$ and Theorem $1.2$ below: 
%%
%%%%%     Theorem 1.2     %%%%%
%%
\begin{thm}
Let $n \ge 2$ be an integer, and $A \in \mathbb{R}\backslash \{0\}$ and $\mu >0$ be constants. 
Assume that 
$$
  |A| < 1 \quad {\rm and} \quad 4 \mu^2 < \frac{A^2}{4-A^2} .
$$
Then, there exist a rotationally symmetric manifold $(\mathbb{R}^n, g:=dr^2+ f^2(r) g_{S^{n-1}(1)})$ and a constant $r_0 > 0$ such that the following {\rm (i)} and {\rm (ii)} hold~{\rm :}
\begin{enumerate}[{\rm (i)}]
\item $\displaystyle \nabla dr = \frac{1}{n-1} \Big\{ 1 + A \frac{\sin (2\mu r)}{r} \Big\} \widetilde{g}$ \ \ for $r \ge r_0${\rm ;} in particular, the following holds{\rm :} 
$$
  r(\Delta_g r - 1 ) = A \sin (2\mu r) \quad {\rm for}~r \ge r_0 ~;~
  \sigma_{\rm ess}(-\Delta_g)=[\frac{1}{4}, \infty) ~;
$$
\item $\sigma_{\rm pp}(-\Delta_g) = \left\{ \frac{1}{4}( 1 + 4\mu ^2) \right\}$.
\end{enumerate}
\end{thm}

In Theorem $1.2$, in order that $(\mathbb{R}^n, g)$ is expanding at infinity in the sense of $(*_1)$, $|A|$ must be smaller than $1$; this condition appears as ``~$2 \min\{ \widetilde{\alpha}_1, 1 \} > A_1 + B_1$'' in Corollary $1.1$; indeed, Theorem $1.2$ corresponds to the case $A=A_1=B_1$ and $\widetilde{\alpha}_1 = \infty$ in Corollary $1.1$. 
The upper part from the bottom $\frac{1}{4}$ of the essential spectrum is for the case $(*_8)$ with $c=1$ and $A=A_1=B_1$ is $\frac{4}{(1-A)^2}A^2$; on the other hand, the upper part from $\frac{1}{4}$ in Theorem $1.2$ is $\frac{1}{4(4-A^2)} A^2$; that is, upper parts from $\frac{1}{4}$ coincides with $A^2$ in both cases up to constant multipliers. In this sense, $(*_8)$ seems to be optimum. 

Theorem $1.2$ can be proved by slightly modifying the proof of Theorem $1.8$ in \cite{K2}. 

Corollary $1.1$ immediately implies the following corollary. 
The assumptions of Corollary 1.2 are very simple, comparing Corollary $1.1$. 

%%
%%%%%     Corollary 1.2     %%%%%
%%
\begin{cor}
Let $(M,g)$ be a noncompact connected complete Riemannian manifold, $E$ be an end with radial coordinates of $(M,g)$, and $w$ be a $C^{\infty}$-function on $M$. 
Let $r$ denote ${\rm dist}\,(\partial E,*)$ on $E$. 
Assume that there exist constants $\widetilde{\alpha}_1 > 0$, $r_0 \ge 0$, $b \in \mathbb{R}$, and $c \in \mathbb{R}$ such that 
\begin{align*}
 \nabla dr \ge \frac{\widetilde{\alpha}_1}{r} \,\widetilde{g} 
 \quad {\rm on}~E(r_0,\infty); \ 
 \Delta_g r + \frac{\partial w}{\partial r} - c - \frac{b}{r} = o(r^{-1}) 
 \quad {\rm on}~E(r_0,\infty). 
\end{align*}
Let $L_{M \backslash E}$ denote the Dirichlet drift Laplacian $\Delta_g + \nabla w$ on $L^2(M \backslash E, e^{w} v_g)$, and assume that $c \neq 0$ and $\min \sigma_{\rm ess}\left( L_{M \backslash E} \right) \ge \frac{c^2}{4}$. 
Then, the drift Laplacian $-\Delta_g - \nabla w$ on $L^2(M, e^{w} v_g)$ satisfies $\sigma_{\rm ess}(-\Delta_g - \nabla w) = [\frac{c^2}{4}, \infty)$ and $\sigma_{\rm pp}(-\Delta_g -\nabla w) \cap (\frac{c^2}{4}, \infty) = \emptyset$. 
\end{cor}

In case $c = 0$, Corollary $1.1$ immediately implies the following: 

%%
%%%%%     Corollary 1.3     %%%%%
%%
\begin{cor}
Let $(M,g)$ be a noncompact connected complete Riemannian manifold, $E$ be an end with radial coordinates of $(M,g)$, and $w$ be a $C^{\infty}$-function on $M$. 
Let $r$ denote ${\rm dist}\,(\partial E,*)$ on $E$. 
Assume that there exist constants $\widetilde{\alpha}_1 > 0$, $A_1 > 0$, $B_1 > 0$, $r_0 \ge 0$, and $b \in \mathbb{R}$ such that 
\begin{align*}
  \nabla dr \ge \frac{\widetilde{\alpha}_1}{r} \,\widetilde{g} 
   \quad {\rm on}~E(r_0,\infty); \ 
  -A_1 \le r \Big(\Delta_g r + \frac{\partial w}{\partial r} \Big) - b 
  \le B_1 \quad {\rm on}~E(r_0,\infty). 
\end{align*}
Assume that $2 \min\{ \widetilde{\alpha}_1, 1 \} > A_1 + B_1$. 
Then, $\sigma_{\rm ess}(-\Delta_g - \nabla w) = [ 0, \infty)$ and $\sigma_{\rm pp}(-\Delta_g -\nabla w) = \emptyset$. 
\end{cor}

In Corollary 1.3, the amplitude of the constant $A_1 + B_1$ cannot be larger than $4$. 
Indeed, if $A_1 + B_1 > 4$, an embedded eigenvalue may emerge, as Theorem $1.3$ below shows. 
%%
%%%%%     Theorem 1.3     %%%%%
%%
\begin{thm}
Let $n \ge 2$ be an integer, and $A \in \mathbb{R}$, $\mu >0$, and $b > 0$ be constants. 
Assume that 
$$
  b > |A| > 2.
$$
Then, there exist a rotationally symmetric manifold $(\mathbb{R}^n, g:=dr^2+ f^2(r) g_{S^{n-1}(1)})$ and a constant $r_0 > 0$ such that the following {\rm (i)} and {\rm (ii)} hold~{\rm :}
\begin{enumerate}[{\rm (i)}]
\item $\nabla dr = \frac{b - A \sin (2\mu r)}{(n-1)r}\,\widetilde{g}$ \ for $r \ge r_0$. 
In particular, 
\begin{align*}
& \nabla dr \ge \frac{b-|A|}{r} \widetilde{g} \quad {\rm for}~r \ge r_0 ~; 
  \quad \Delta_g r 
  = \frac{b}{r} - A \frac{\sin (2\mu r)}{r} \quad {\rm for}~r \ge r_0; \\ 
& \sigma_{\rm ess}(-\Delta_g) = [0, \infty);
\end{align*}
\item $\mu ^2 \in \sigma_{\rm pp}(-\Delta_g)$.
\end{enumerate}
\end{thm}

Theorem $1.3$ can be proved by slightly modifying the proof of Theorem $1.8$ in \cite{K2}. 

By putting $w\equiv 0$ in Corollary $1.1$, we obtain the following:

%%
%%%%%     Corollary 1.4     %%%%%
%%
\begin{cor}
Let $(M, g)$ be a noncompact connected complete Riemannian manifold and $E$ be an end with radial coordinates of $(M,g)$. 
Let $r$ denote ${\rm dist}\,(\partial E,*)$ on $E$. 
Assume that there exist constants $\widetilde{\alpha}_1 > 0$, $A_1 > 0$, $B_1 > 0$, $r_0 \ge 0$, $b \in \mathbb{R}$, $c \ge 0$ such that 
\begin{align*}
 \nabla dr \ge \frac{\widetilde{\alpha}_1}{r} \,\widetilde{g} 
 \quad {\rm on}~E(r_0,\infty). \ 
 -A_1 \le r \Big(\Delta_g r - c \Big) - b 
 \le B_1 \quad {\rm on}~E(r_0,\infty). 
\end{align*}
Assume that $ 2 \min \{ \widetilde{\alpha}_1, 1 \} > A_1 + B_1$. 
Then, $\sigma_{\rm ess}(-\Delta_g ) \supseteq [\frac{c^2}{4}, \infty)$ and  
\begin{align*}
 \sigma_{\rm pp}(-\Delta_g ) \cap \left( \frac{c^2}{4} \bigg\{ 1 
 + \frac{(A_1 + B_1)^2}{(2 - \varepsilon_1 - B_1)(\varepsilon_1 - A_1)} \bigg\}, ~ \infty \right) = \emptyset, 
\end{align*}
where $\varepsilon_1$ is the constant defined in Corollary $1.1$. 
\end{cor}
Corollary $1.4$ is a generalization of results in author's earlier papers \cite{K2} and \cite{K3}.

Theorem $1.1$ will be obtained by modifying and {\it strengthening} the arguments in \cite{K3}. 

\vspace{1mm}

This paper is organized as follows. 
Section $2$, $3$, $4$, and $5$ are devoted to the proof of Theorem $1.1$. 
Section $6$ is concerned with the relationship between the radiation conditions and the growth condition $(*_3)$; we shall prove Lemma $8.1$ in \cite{K4} there. 
In Section $7$, we shall construct several Riemannian manifolds whose Laplacians satisfy the absence of embedded eigenvalues and besides the absolutely continuity, but their growth orders of metrics on ends are very complicated. 

%%
%%%%%     SECTION 2     %%%%%
%%
\section{Analytic propositions}

In this section, we shall prepare some analytic propositions for the proof of Theorem $1.1$. 

First, let $c \in \mathbb{R}$ be a constant; we shall transform the operator $\Delta_g + \nabla w + \frac{c^2}{4}$ and the measure $e^w v_g$ into the new operator $L := \exp (\frac{c}{2}r) \circ (\Delta_g + \nabla w) \circ \exp (-\frac{c}{2}r) $ and new measure $e^{-cr+w} dv_g$, respectively: 
\begin{align*}
 \begin{CD}
L^2(E, e^w v_g) @>{ -(\Delta_g + \nabla w + \frac{c^2}{4})}>> L^2(E, e^w v_g) \\
 @V{\exp(\frac{c}{2}r)}VV      @VV{\exp(\frac{c}{2}r)}V  \\
L^2(E, e^{-cr+w} v_g)   @>>{ - L }>     L^2(E, e^{-cr+w} v_g)
 \end{CD}
\end{align*}
Here, note that the multiplying operator $\exp(\frac{c}{2}r) : L^2(E, e^w v_g) \ni h \mapsto \exp(\frac{c}{2}r) h \in L^2(E, e^{-cr+w} v_g)$ is a unitary operator. 

Then, note the following: 
%%
%%%%%     LEMMA 2.1     %%%%%
%%
\begin{lem}
Let $\gamma >0$ be a constant and $u$ be a $C^{\infty}$-function on $E$. 
We shall set $h(x):=\exp \left(-\frac{c}{2}r(x)\right) u(x)$ for $x \in E$. 
Then, the following conditions {\rm (i)} and {\rm (ii)} are equivalent~{\rm :}
\begin{enumerate}[{\rm (i)}]
\item $\displaystyle \liminf_{R\to \infty} R^{\gamma} \int_{S(R)} \Big\{ \bigg(\frac{\partial u}{\partial r} \bigg)^2 + u^2 \Big\} e^{-cr+w} dA = 0;$
\item $\displaystyle\liminf_{R\to \infty} R^{\gamma} \int_{S(R)} \Big\{ \bigg(\frac{\partial h}{\partial r} \bigg)^2 + h^2 \Big\} e^{w} dA = 0$.
\end{enumerate}
\end{lem}
\begin{proof}
Direct computations show that 
\begin{align*}
  \Big\{ \bigg( \frac{\partial u}{\partial r} \bigg)^2 + u^2 \Big\} e^{-cr+w}  
= \Big\{ \bigg( \frac{\partial h}{\partial r} \bigg)^2 + c h \frac{\partial h}{\partial r} + \Big( \frac{c^2}{4} + 1 \Big) h^2 \Big\} e^{w}.
\end{align*}
If $c=0$, the assertion is trivial. 
Hence, assume that $c\neq 0$. 
Then, in general, there exists a constant $c_0(c)>0$, depending only on $c$, such that $X^2 + c XY + (\frac{c^2}{4} + 1)Y^2 \ge c_0(c)\{X^2+Y^2 \}$ holds for any $X,Y \in \mathbb{R}$. 
Therefore, (i) implies (ii). 
The contrary is proved in the same manner. 
\end{proof}
From Lemma $2.1$, we see that it suffices to prove Theorem $1.1$ for $-L$, $e^{-cr+w} A$, and $e^{-cr+w} v_g$ in stead of $-( \Delta_g + \nabla w + \frac{c^2}{4})$, $e^w A$, and $e^w v_g$, respectively.

Now, let $\lambda > 0$ be a constant and $u$ be a solution of 
\begin{align*}
  Lu + \lambda u= 0 \quad {\rm on}~~E,
\end{align*}
and assume that (i) in Lemma $2.1$ holds. 
A direct computation shows that 
\begin{align}
& \Delta_g u + \langle \nabla w - c \nabla r, \nabla u \rangle 
  - \frac{c}{2} q_{\star} u + \lambda u = 0 \quad {\rm on}~~E; \\
& q_{\star} := \Delta_g r + \frac{\partial w}{\partial r} - c.
\end{align}
Let $\rho (t)$ be a $C^{\infty}$ function of $t \in [r_0,\infty)$, and put 
\begin{align*}
  v(x) := \exp \bigl(\rho (r(x)) \bigr)u(x) \qquad {\rm for}~~x \in E.
\end{align*}
Then, direct computations show that
\begin{align}
& \Delta_g v - \big( 2 \rho '(r) + c \big) \frac{\partial v}{\partial r} 
  + \langle \nabla w, \nabla v \rangle 
  + \Big\{ q_1 - \Big( \rho'(r) + \frac{c}{2} \Big)q_{\star} + \lambda \Big\} 
  v = 0,\\
& q_1 := -\rho''(r) + (\rho'(r))^2. \notag
\end{align}

In order to prove Theorem $1.1$, we shall prepare three Propositions below:
%%
%%%%%     PROPOSITION 2.1     %%%%%
%%
\begin{prop}
For any $\psi \in C^{\infty}(E)$ and $r_0 \le s< t$, we have
\begin{align*}
& \int_{E(s, t)} \psi \left\{ |\nabla v|^2 - 
  \Big( \lambda + q_1 - (\rho'(r) + \frac{c}{2}) q_{\star} \Big) v^2 \right\} 
  e^{-cr+w} \,dv_g \\
=
& \left( \int_{S(t)}-\int_{S(s)}\right ) 
  \psi \frac{\partial v}{\partial r} v e^{-cr+w} \,dA 
  - \int_{E(s, t)} \langle \nabla \psi + 2 \psi \rho '(r) \nabla r , \nabla v 
  \rangle v e^{-cr+w} \,dv_g.
\end{align*}
\end{prop}
\begin{proof}
We shall multiply the equation $(3)$ by $\psi v$ and integrate it over $E(s, t)$ with respect to the measure $e^{-cr+w} v_g$.
Then, the Green's formula yields Proposition $2.1$. 
\end{proof}
%%
%%%%%     PROPOSITION 2.2     %%%%%
%%
\begin{prop}
For any $r_0 \le s< t$ and $\gamma \in \mathbb{R}$, we have
\begin{align*}
& \int_{S(t)} r^{\gamma} \bigg\{ \bigg(\frac{\partial v}{\partial r} \bigg)^2
  - \frac{1}{2}|\nabla v|^2 + \frac{1}{2}(\lambda + q_1)v^2 \bigg\} 
  e^{-cr+w}\,dA \\
& + \int_{S(s)} r^{\gamma} \bigg\{ \frac{1}{2}|\nabla v|^2 - \bigg( \frac{\partial v}{\partial r} \bigg)^2 - \frac{1}{2}(\lambda + q_1)v^2 \bigg\} 
  e^{-cr+w} \,dA \\
= 
& \int_{E(s, t)} r^{\gamma -1} \Big\{ r(\nabla dr)(\nabla v,\nabla v) 
  - \frac{1}{2} ( \gamma + rq_{\star} ) \widetilde{g}(\nabla v, \nabla v) \Big\}  e^{-cr+w} \, dv_g \\
& + \int_{E(s, t)} r^{\gamma -1} \Big\{ \frac{1}{2} ( \gamma - rq_{\star} ) 
  + 2 \rho'(r) r \Big\} \bigg( \frac{\partial v}{\partial r} \bigg)^2 
  e^{-cr+w} \, dv_g \\
& + \int_{E(s, t)} r^{\gamma -1} ~ rq_{\star}~ \Big( \rho'(r) + \frac{c}{2} \Big) \frac{\partial v}{\partial r} v e^{-cr+w} \, dv_g \\
& + \frac{1}{2} \int_{E(s, t)} r^{\gamma -1} \Big\{ (\lambda + q_1)(\gamma + rq_{\star}) + r \frac{\partial q_1}{\partial r} \Big\} v^2 e^{-cr+w} \, dv_g .
\end{align*}
\end{prop}
\begin{proof}
We shall multiply the equation $(3)$ by $\langle \nabla r, \nabla v \rangle$. 
Then, from 
\begin{align*}
  \langle \nabla r, \nabla v \rangle \Delta_g v 
& = \langle \nabla r, \nabla v \rangle {\rm div}(\nabla v) \\
& = {\rm div} \big( \langle \nabla r, \nabla v \rangle \nabla v \big) 
   - \langle \nabla_{\nabla v} (\nabla r), \nabla v \rangle 
   - \langle \nabla r, \nabla_{\nabla v} (\nabla v) \rangle ;\\
 \langle \nabla r, \nabla_{\nabla v} (\nabla v) \rangle 
& = \langle \nabla_{\nabla r} (\nabla v), \nabla v \rangle 
  = \frac{1}{2} (\nabla r) (|\nabla v|^2) \\ 
& = \frac{1}{2} {\rm div} (|\nabla v|^2 \nabla r) 
    - \frac{1}{2} |\nabla v|^2 \Delta_g r,
\end{align*}
we have
\begin{align*}
  \langle \nabla r, \nabla v \rangle \Delta_g v 
= {\rm div} \Big( \langle \nabla r, \nabla v \rangle \nabla v - \frac{1}{2} |\nabla v|^2 \nabla r \Big) - (\nabla dr)(\nabla v,\nabla v) + \frac{1}{2} |\nabla v|^2 \Delta_g r .
\end{align*}
Therefore, we obtain 
\begin{align*}
& {\rm div} \bigg( \frac{\partial v}{\partial r} \nabla v 
  - \frac{1}{2} |\nabla v|^2 \nabla r \bigg) - (\nabla dr)(\nabla v,\nabla v) 
  + \frac{1}{2} |\nabla v|^2 \Delta_g r - (2\rho'(r) + c) 
  \bigg( \frac{\partial v}{\partial r} \bigg)^2 \notag \\
& \quad + \langle \nabla w, \nabla v \rangle \frac{\partial v}{\partial r} 
  + \Big\{ \lambda -\rho''(r) + (\rho'(r))^2 
  - \Big( \rho'(r) + \frac{c}{2} \Big) q_{\star} \Big\} 
  v\frac{\partial v}{\partial r} =0 . 
\end{align*}
We shall multiply the equation above by $r^{\gamma} e^{-cr+w}$ and use a general formula, $f{\rm div}X= {\rm div}(fX) - Xf$. 
After that, integrating it over $E(s, t)$ with respect to $v_g$, we obtain, by the divergence theorem, 
\begin{align}
& \left( \int_{S(t)} - \int_{S(s)} \right) r^{\gamma} 
  \bigg\{ \bigg( \frac{\partial v}{\partial r} \bigg)^2 
  - \frac{1}{2}|\nabla v|^2 \bigg\} e^{-cr+w} \, dA \notag \\
=
& \int_{E(s, t)} r^{\gamma -1} \bigg\{ r(\nabla dr)(\nabla v, \nabla v) 
  + \gamma \bigg( \frac{\partial v}{\partial r} \bigg)^2 \bigg\} 
  e^{-cr+w} \, dv_g \notag \\
& - \frac{1}{2} \int_{E(s, t)} r^{\gamma -1} \big\{ \gamma + rq_{\star} \big\} 
  |\nabla v|^2 e^{-cr+w} \, dv_g \notag \\
& + 2 \int_{E(s, t)} r^{\gamma} \rho'(r) 
  \bigg( \frac{\partial v}{\partial r} \bigg)^2 e^{-cr+w} \, dv_g \notag \\
& + \int_{E(s, t)} r^{\gamma -1} \Big( \rho'(r) + \frac{c}{2} \Big) ~ 
  rq_{\star} ~ \frac{\partial v}{\partial r} v e^{-cr+w} \, dv_g \notag \\
& - \int_{E(s, t)} r^{\gamma} \{ \lambda + q_1 \} \frac{\partial v}{\partial r} 
  v e^{-cr+w} \, dv_g. 
\end{align}
Here, the integrand of the last term of $(4)$ is equal to 
\begin{align*}
& - \frac{1}{2} r^{\gamma} \{ \lambda + q_1 \} 
  \langle \nabla r, \nabla (v)^2 \rangle e^{-cr+w} \\
= 
& - \frac{1}{2} {\rm div} \Big( r^{\gamma} (\lambda + q_1) v^2 e^{-cr+w} \nabla r \Big) \\
& + \frac{1}{2} r^{\gamma-1} \Big\{ (\lambda + q_1) ~ rq_{\star} + \gamma (\lambda + q_1) + r \frac{\partial q_1}{\partial r} \Big\} v^2 e^{-cr+w}.
\end{align*}
Hence, integrating this equation over $E(s, t)$ with respect to $v_g$, we have 
\begin{align}
& - \int_{E(s, t)} r^{\gamma} \{ \lambda + q_1 \} \frac{\partial v}{\partial r} 
  v e^{-cr+w} \, dv_g \notag \\
=
& - \frac{1}{2} \left( \int_{S(t)} - \int_{S(s)} \right) r^{\gamma} (\lambda + q_1) v^2 e^{-cr+w} dA \notag \\
& + \frac{1}{2} \int_{E(s, t)} r^{\gamma - 1} \Big\{ (\lambda + q_1)( \gamma + rq_{\star}) + r \frac{\partial q_1}{\partial r} \Big\} v^2 e^{-cr+w} \, dv_g. 
\end{align}
Thus, substituting $(5)$ into $(4)$, we obtain Proposition $2.2$. 
\end{proof}
%%
%%%%%     PROROSITION 2.3     %%%%%
%%
\begin{prop}
For any $\gamma \in \mathbb{R}$, $\varepsilon \in \mathbb{R}$, and $0\le s<t$, we have
\begin{align*}
&\int_{S(t)} r^{\gamma} \bigg\{ \bigg( \frac{\partial v}{\partial r} \bigg)^2 
 + \frac{1}{2}(\lambda + q_1)v^2 - \frac{1}{2} |\nabla v|^2 
 + \frac{\gamma - \varepsilon + b}{2r} \frac{\partial v}{\partial r} v \bigg\} 
 e^{-cr+w} \, dA \\
&+ \int_{S(s)} r^{\gamma} \bigg\{ \frac{1}{2} |\nabla v|^2 
 - \frac{1}{2}(\lambda + q_1)v^2 
 - \bigg( \frac{\partial v}{\partial r} \bigg)^2   
 - \frac{\gamma - \varepsilon + b}{2r} \frac{\partial v}{\partial r} v \bigg\} 
 e^{-cr+w} \, dA \\
= 
&\int_{E(s, t)} r^{\gamma -1} \bigg\{ r(\nabla dr)(\nabla v,\nabla v) 
 - \frac{1}{2} ( \varepsilon + r q_{\star} - b ) 
 \widetilde{g}(\nabla v, \nabla v) \bigg\} e^{-cr+w} \, dv_g \\
&+ \int_{E(s, t)} r^{\gamma -1} \bigg\{ \gamma 
 - \frac{1}{2} ( \varepsilon + rq_{\star} - b ) + 2 \rho'(r) \,r \bigg\} 
 \bigg( \frac{\partial v}{\partial r} \bigg)^2 e^{-cr+w} \, dv_g \\
&+ \int_{E(s, t)} r^{\gamma -1} \bigg\{ (\gamma - \varepsilon + b)\rho'(r) 
 + \frac{(\gamma - 1)(\gamma - \varepsilon + b)}{2r} 
 + \Big( \rho'(r) + \frac{c}{2} \Big) rq_{\star} \bigg\} 
 \frac{\partial v}{\partial r} v e^{-cr+w} \, dv_g \\
&+ \frac{1}{2} \int_{E(s, t)} r^{\gamma -1} \bigg\{ 
 (\lambda + q_1)(\varepsilon + rq_{\star} - b) + r \frac{\partial q_1}{\partial r}   + (\gamma - \varepsilon + b) \Big( \rho'(r) + \frac{c}{2} \Big) q_{\star} \bigg\}
 v^2 e^{-cr+w} \, dv_g .
\end{align*}
\end{prop}
\begin{proof}
Substitute $\psi = \frac{\gamma-\varepsilon+b}{2} r^{\gamma -1}$ into the equation in Proposition $2.1$ and adding it to the equation in Proposition $2.2$, we obtain Proposition $2.3$. 
\end{proof}
%%
%%%%%     LEMMA 2.2     %%%%%
%%
\begin{lem}
For any $\beta \in \mathbb{R}$, we have 
\begin{align*}
& \int_{S(t)} r^{\beta} v^2 e^{-cr+w} \,dA - \int_{S(s)} r^{\beta} v^2 e^{-cr+w} \, dA \notag \\
= 
& \int_{E(s, t)} r^{\beta} \bigg\{ \Big( q_{\star} + \frac{\beta}{r} \Big) v^2
   + 2 v \frac{\partial v}{\partial r} \bigg\} e^{-cr+w} \,dv_g.
\end{align*}
\end{lem}
\begin{proof}
A direct computation shows that
\begin{align*}
  {\rm div}( r^{\beta} v^2 e^{-cr+w} \nabla r)
= r^{\beta} \Big\{ \Big( q_{\star} + \frac{\beta}{r} \Big) v^2
  + 2v\frac{\partial v}{\partial r} \Big\} e^{-cr+w}.
\end{align*}
Integrating this equation with respect to $v_g$ over $E(s, t)$, we obtain Lemma $2.2$. 
\end{proof}
%%
%%%%%     SECTION 3     %%%%%
%%
\section{Faster than polynomial decay}
The proof of Theorem $1.1$ will be accomplished by following three procedures: (1) to show faster than polynomial decay; (2) to show faster than exponential decay; (3) to show vanishing on a neighborhood of infinity. 
Section $3$, $4$, and $5$ will be devoted to these procedures $(1)$, $(2)$, and $(3)$, respectively. 
%%
%%%%%     Theorem 3.1     %%%%%
%%
\begin{thm}
Let $(M, g)$ be a noncompact Riemannian manifold and $E$ be an end with radial coordinates of $(M,g)$. 
Let $r$ denote ${\rm dist}(\partial E, *)$ on $E$. 
Assume that there exist constants $\widetilde{\alpha}_1 > 0$, $A_1 > 0$, $B_1 > 0$, $r_0 \ge 0$, $b \in \mathbb{R}$, and $c \in \mathbb{R}$ such that 
\begin{align}
& \nabla dr \ge \frac{\widetilde{\alpha}_1}{r} \,\widetilde{g} 
  \quad {\rm on}~E(r_0,\infty); \\
& -A_1 \le r \Big(\Delta_g r + \frac{\partial w}{\partial r} - c \Big) - b 
  \le B_1 \quad {\rm on}~E(r_0,\infty). 
\end{align}
Let $\lambda > 0$ and $\gamma >0 $ be constants, and assume that $u$ is a solution of 
\begin{align*}
  L u + \lambda u = 0 \quad {\rm on}~~E,
\end{align*}
satisfying 
\begin{align}
 \liminf_{t \to \infty} t^{\gamma} \int_{S(t)} 
 \bigg\{ \bigg( \frac{\partial u}{\partial r} \bigg)^2 + u^2 \bigg\} 
 e^{-cr+w} dA = 0.
\end{align}
Assume that 
\begin{align}
 2 \min \{ \widetilde{\alpha}_1, \gamma \} > A_1 + B_1~; \quad 
 \lambda > \frac{c^2 (A_1 + B_1)^2}{4(2\gamma - \varepsilon_0 - B_1)(\varepsilon_0 - A_1)},  
\end{align}
where $\varepsilon_0$ is the constant defined by $(*_4)$. 
Then, we have, for any $m>0$,
\begin{equation}
 \int_{E(r_0,\infty)} r^m \left\{ |\nabla u|^2 + |u|^2 \right\} e^{-cr+w} 
 \, dv_g < \infty.
\end{equation}
\end{thm}
\begin{proof}
Considering the first condition of $(9)$, we shall take a constant $\varepsilon$ so that
\begin{align}
 2 \min \{ \widetilde{\alpha}_1, \gamma \} - B_1 > \varepsilon > A_1.
\end{align}
We shall put $\rho(r)=0$ in Proposition $2.3$. 
Then, $v = u$ and $q_1=0$. 
Moreover, in view of $(2)$, the assumptions $(6)$ and $(7)$ implie that
\begin{align*}
  r (\nabla dr) (\nabla u, \nabla u) 
  - \frac{1}{2}(\varepsilon + rq_{\star} - b) 
  \, \widetilde{g}(\nabla u, \nabla u) 
\ge 
  \frac{1}{2} \Big\{ 2 \widetilde{\alpha}_1 - B_1 - \varepsilon \Big\} \, 
  \widetilde{g}(\nabla u, \nabla u) . 
\end{align*}
For simplicity, we shall set 
\begin{align*}
  c_{\max} := 
 \begin{cases}
\ c A_1   \qquad & {\rm if} \quad c > 0 , \\
\ \  0    \qquad & {\rm if} \quad c = 0 , \\
\ -c B_1  \qquad & {\rm if} \quad c < 0 .
 \end{cases}
\end{align*}
Then, we have $-c ( r q_{\star} - b) \le c_{\max}$, and hence, $c r q_{\star} + c_{\max} - cb \ge 0$.
Therefore, we obtain, for $r_0 \le s <t$, 
\begin{align}
& \int_{S(t)} r^{\gamma} \bigg\{ 2 \bigg( \frac{\partial u}{\partial r} \bigg)^2  + \lambda u^2 - |\nabla u|^2 
  + \frac{\gamma - \varepsilon + b}{r} \frac{\partial u}{\partial r} u \bigg\} 
  e^{-cr+w} \,dA \notag \\
& + \int_{S(s)} r^{\gamma} \bigg\{ |\nabla u|^2 - \lambda u^2 
  + \frac{c}{2} q_{\star} u^2 + \frac{c_{\max} - cb}{2r} u^2 
  - 2 \bigg( \frac{\partial u}{\partial r} \bigg)^2 
  - \frac{\gamma - \varepsilon + b}{r} \frac{\partial u}{\partial r} u \bigg\} 
  e^{-cr+w} \,dA \notag \\
\ge 
& \int_{S(t)} r^{\gamma} \bigg\{ 2 \bigg( \frac{\partial u}{\partial r} \bigg)^2  + \lambda u^2 - |\nabla u|^2 
  + \frac{\gamma - \varepsilon + b}{r} \frac{\partial u}{\partial r} u \bigg\} 
  e^{-cr+w} \,dA \notag \\
& + \int_{S(s)} r^{\gamma} \bigg\{ |\nabla u|^2 - \lambda u^2 
  - 2 \bigg( \frac{\partial u}{\partial r} \bigg)^2 
  - \frac{\gamma - \varepsilon + b}{r} \frac{\partial u}{\partial r} u \bigg\} 
  e^{-cr+w} \,dA \notag \\
\ge 
& \int_{E(s, t)} r^{\gamma-1} 
  \big\{ 2 \widetilde{\alpha}_1 - B_1 - \varepsilon \big\} \, 
  \widetilde{g}(\nabla u, \nabla u) e^{-cr+w} \,dv_g \notag \\
& + \int_{E(s, t)} r^{\gamma-1} 
  \big\{ 2 \gamma - \varepsilon - B_1 \big\} 
  \bigg( \frac{\partial u}{\partial r} \bigg)^2 e^{-cr+w} \,dv_g \notag \\
& + \int_{E(s, t)} r^{\gamma-1} 
  \big\{ c r q_{\star} + O(r^{-1}) \big\} 
  \frac{\partial u}{\partial r}u e^{-cr+w} \,dv_g \notag \\
& + \int_{E(s, t)} r^{\gamma-1} 
 \big\{ \lambda (\varepsilon - A_1) + O(r^{-1}) \big\} u^2 e^{-cr+w} \,dv_g. 
\end{align}
Let $<\alpha\ll 1$ be a small constant determined later. 
Substituting $\beta = \gamma -1$ into the equation in Lemma $2.2$, and multiplying it by the constant $\frac{c_{\max} - cb}{2} + \alpha$, we obtain 
\begin{align}
& \Big( \frac{c_{\max} - cb}{2} + \alpha \Big) \int_{S(t)} r^{\gamma-1} u^2 
  e^{-cr+w} \,dA  
  - \Big( \frac{c_{\max} - cb}{2} + \alpha \Big) \int_{S(s)} r^{\gamma-1} u^2 
  e^{-cr+w} \,dA \notag \\
= 
& \int_{E(s, t)} r^{\gamma-1} \Big\{ O(r^{-1}) u^2 
  + (c_{\max} - cb + 2\alpha) \frac{\partial u}{\partial r}u \Big\} 
  e^{-cr+w} \,dv_g.
\end{align}
Addition of $(13)$ to $(12)$ yields 
\begin{align}
& \int_{S(t)} r^{\gamma} \bigg\{ 
  2 \bigg( \frac{\partial u}{\partial r} \bigg)^2 
  + \big( \lambda + O(r^{-1}) \big) u^2 
  + \frac{\gamma - \varepsilon + b}{r} \frac{\partial u}{\partial r} u \bigg\} 
  e^{-cr+w} \,dA \notag \\
& + \int_{S(s)} r^{\gamma} \bigg\{ |\nabla u|^2 - \lambda u^2 
  + \frac{c}{2} q_{\star} u^2 - \frac{\alpha}{r} u^2 
  - 2 \bigg( \frac{\partial u}{\partial r} \bigg)^2 
  - \frac{\gamma - \varepsilon + b}{r} \frac{\partial u}{\partial r} u \bigg\} 
  e^{-cr+w} \,dA \notag \\
\ge 
& \int_{E(s, t)} r^{\gamma-1} \Big\{ 2 \widetilde{\alpha}_1 - \varepsilon 
  - B_1 \Big\} \, \widetilde{g}(\nabla u, \nabla u) e^{-cr+w} \,dA \notag \\
& + \int_{E(s, t)} r^{\gamma-1} 
  \big\{ 2 \gamma - \varepsilon - B_1 \big\} 
  \bigg( \frac{\partial u}{\partial r} \bigg)^2 e^{-cr+w} \,dv_g \notag \\
& - \int_{E(s, t)} r^{\gamma-1} 
  \big\{ |c| ( A_1 + B_1 ) + 2 \alpha + O(r^{-1}) \big\} 
  \bigg| \frac{\partial u}{\partial r} u \bigg| e^{-cr+w} \,dv_g \notag \\
& + \int_{E(s, t)} r^{\gamma-1} 
 \big\{ \lambda (\varepsilon - A_1) + O(r^{-1}) \big\} u^2 e^{-cr+w} \,dv_g,
\end{align}
where we have used the fact, $|c (r q_{\star} - b) + c_{\max} | \le |c|(A_1+B_1)$. 

The discriminant of a quadratic equation $(2\gamma - \varepsilon - B_1)x^2 - |c|(A_1+B_1)x + \lambda (\varepsilon -A_1)=0$ is equal to $c^2(A_1+B_1)^2 - 4(2\gamma -\varepsilon -B_1)\lambda (\varepsilon - A_1)$; we shall consider the function 
$$
  h(t) := \frac{1}{4(2\gamma - t - B_1)(t - A_1)} 
  \quad {\rm for}~t \in \big\{ t \mid A_1 < t < 2 \gamma - B_1,~2 \widetilde{\alpha}_1 - B_1 \ge t \big\} ;
$$
then, $h(t)$ takes the minimum value at $\varepsilon_0$. 
Hence, in view of the second condition of $(9)$, by taking $\varepsilon < \varepsilon_0$ sufficiently close to $\varepsilon_0$ and taking $\alpha > 0$ sufficiently small in $(14)$, we see that, there exists constants $r_1 = r_1(\lambda, \gamma, \widetilde{\alpha}_1, A_1, B_1, \alpha ) \ge r_0$ and $C_1 = C_1(\lambda, \gamma, \widetilde{\alpha}_1, A_1, B_1, \alpha ) > 0$ such that the right hand side of $(14)$ is bounded from below by   
\begin{align}
  \frac{C_1}{2} \int_{E(s, t)} r^{\gamma-1} \{ |\nabla u|^2 + u^2 \} e^{-cr+w} 
  \,dv_g \qquad {\rm for~any}~~t > s \ge r_1 . 
\end{align}
On the other hand, there exists a constant $r_2 = r_2(\alpha, \varepsilon, b, \gamma)$ such that 
\begin{align}
  - \frac{\alpha}{r} u^2 
  - 2 \bigg( \frac{\partial u}{\partial r} \bigg)^2 
  - \frac{\gamma - \varepsilon + b}{r} \frac{\partial u}{\partial r} u 
 \le 0 \quad {\rm for}~~r\ge r_2 .
\end{align}
Furthermore, the assumption $(8)$ implies that, there exits a divergent sequence $\{t_i\}$ of real numbers such that the first term with $t = t_i$ on the left hand side of $(14)$ converges to zero as $i\to \infty$. 
Hence, taking $(15)$ and $(16)$ into account, putting $t = t_i$ in $(14)$, and letting $i\to \infty$, we obtain, for $t > s \ge r_3 :=\max \{ r_1 , r_2 \}$, 
\begin{align}
& \int_{S(s)}r^{\gamma} \left\{ |\nabla u|^2 - \lambda u^2 
  + \frac{c}{2} q_{\star} u^2 \right\} e^{-cr+w} \,dA \notag \\
\ge
& C_1 \int_{E(s,\infty)} r^{\gamma-1} \left\{ |\nabla u|^2 + u^2 \right\} 
  e^{-cr+w} \,dv_g .
\end{align}
Thus, integrating the both sides of $(17)$ with respect to $s$ over $[t,t_1]$, we have, for $r_3 \le t < t_1$, 
\begin{align*}
& C_1 \int^{t_1}_t \, ds \int_{E(s,\infty)} r^{\gamma-1} 
  \left\{ |\nabla u|^2 + |u|^2 \right\} e^{-cr+w} \,dv_g \\
\le 
& \int_{E(t,t_1)} r^{\gamma}
  \left\{ |\nabla u|^2 - \lambda u^2 + \frac{c}{2} q_{\star} u^2 \right\} 
  e^{-cr+w} \, dv_g \\
= 
& \left( \int_{S(t_1)} - \int_{S(t)} \right) r^{\gamma} 
  \frac{\partial u}{\partial r}u e^{-cr+w} \,dA 
  - \gamma \int_{E(t,t_1)} r^{\gamma-1} \frac{\partial u}{\partial r}u e^{-cr+w}  \,dv_g .
\end{align*}
Here, in the last line, we have used the equation in Proposition $2.1$ with $\rho (r) = 0$ and $\psi = r^{\gamma}$. 
Since $(8)$ implies 
\begin{align*}
  \liminf_{t_1\to \infty} \int_{S(t_1)} r^{\gamma}
  \frac{\partial u}{\partial r}u e^{-cr+w} \,dA = 0,
\end{align*}
letting appropriately $t_1\to \infty$ and using Fubini's theorem, we obtain, from the inequality above, 
\begin{align}
& C_1 \int^{\infty}_t \,ds \int_{E(s,\infty)} r^{\gamma-1} 
  \big\{ |\nabla u|^2 + u^2 \big\} e^{-cr+w} \,dv_g \notag \\
= 
& C_1 \int_{E(t,\infty)} (r - t) r^{\gamma-1}
  \big\{ |\nabla u|^2 + u^2 \big\} e^{-cr+w} \,dv_g \notag \\
\le 
& \int_{S(t)} r^{\gamma}
  \bigg\{ \left( \frac{\partial u}{\partial r} \right)^2 + u^2 \bigg\} 
  e^{-cr+w} \,dA
  + \gamma \int_{E(t,\infty)} r^{\gamma-1} 
  \bigg\{ \left( \frac{\partial u}{\partial r} \right)^2 + u^2 \bigg\}
  e^{-cr+w} \,dv_g \notag \\ 
<
& \infty . 
\end{align}
Here note that the right hand side of $(18)$ is finite by $(17)$. 
Thus, we see that the desired assertion $(10)$ holds for $m = \gamma$. 

Integrating $(18)$ with respect to $t$ over $[t_1,\infty)$, and using Fubini's theorem, we obtain, for $t_1 \ge r_1$, 
\begin{align*}
& C_1 \int_{E(t_1, \infty)} (r - t)^2 r^{\gamma - 1} 
  \big\{ |\nabla u|^2 + u^ 2 \big\} e^{-cr+w} \,dv_g \\
\le 
& \int_{E(t_1, \infty)} r^{\gamma} 
  \bigg\{ \left( \frac{\partial u}{\partial r} \right)^2 + u^2 \bigg\} 
  e^{-cr+w} \,dv_g \\ 
& + \gamma \int_{E(t_1, \infty)} (r - t) r^{\gamma - 1}
  \bigg\{ \left( \frac{\partial u}{\partial r} \right) ^2 + u^2 \bigg\} 
  e^{-cr+w} \,dv_g \\ 
< & \infty,
\end{align*}
where, note that the right hand side of this inequality is finite by $(18)$. 
Thus, we see that the desired assertion $(10)$ holds for $m = \gamma + 1$. 
Repeating the integration with respect to $t$ shows that the assertion $(10)$ is valid for $m = \gamma + 2, \gamma + 3, \cdots $, therefore, for any $m > 0$. 
\end{proof}
%%
%%%%%     Section 4     %%%%%
%%
%%
\section{Faster than exponential decay}
We shall first prove Lemma $4.1$ and Lemma $4.2$, which will be used in the proof of Theorem $4.1$: 
%%
%%%%%     Lemma 4.1     %%%%%
%%
\begin{lem}
Assume that the conditions in Theorem $3.1$ holds. 
Assume that, there exists a constant $k_0 \ge 0$ such that $(22)$ below holds, and set $\rho (r) := k_0 + m \log r$. 
Then, for $x \ge r_0$, $v= r^m e^{k_0r} u$ satisfies 
\begin{align*}
& \int_{E(x, \infty)} r^{1 - 2m} \Big\{ |\nabla v|^2 
  - \Big( \lambda + q_1 - \big( \rho '(r) + \frac{c}{2} \big)q_{\star} \Big) v^2 \Big\} 
  e^{-cr+w} \, dv_g \\ 
=
& - \frac{1}{2} \frac{d}{dx}\bigg( x^{1-2m} \int_{S(x)} v^2 e^{-cr+w} 
  \, dA \bigg) 
  - \frac{1}{2} \int_{S(x)} r^{-2m} \big\{ 2m - 1 - rq_{\star} \big\} v^2 
  e^{-cr+w} \, dA \\
& - \int_{E(x, \infty)} r^{-2m} \frac{\partial v}{\partial r} v 
  e^{-cr+w} \, dv_g. 
\end{align*}
\end{lem}
\begin{proof}
Let $A_{\partial E}$ denote the induced measure on $\partial E$, and write $A= \sqrt{G} \, A_{\partial E}$ on $E$. 
Then, a direct computation shows that 
\begin{align}
& \frac{d}{dx}\bigg( x^{1-2m} \int_{S(x)} v^2 e^{-cr+w} \, dA \bigg) 
  \notag \\
= 
& \int_{S(x)} r^{-2m} \big\{ 1 - 2m + rq_{\star} \big\} v^2 e^{-cr+w} \, dA 
  + 2 \int_{S(x)} r^{1-2m} \frac{\partial v}{\partial r} v e^{-cr+w} \, dA, 
\end{align}
where we have used the definition $(2)$ of $q_{\star}$ and the equation $\frac{\partial \sqrt{G}}{\partial r} = (\Delta_g r) \sqrt{G}$. 

On the other hand, we shall substitute $\psi = r^{1-2m}$ into the equation in Proposition $2.1$. 
Then, we have, for $r_0 \le x < t$, 
\begin{align*}
& \int_{E(x, t)} r^{1 - 2m} \Big\{ |\nabla v|^2 
  - \Big( \lambda + q_1 - \big( \rho '(r) + \frac{c}{2} \big)q_{\star} \Big) v^2  \Big\} e^{-cr+w} \, dv_g \\ 
= 
& \int_{S(t)} r^{1 - 2m} \frac{\partial v}{\partial r} v e^{-cr+w} \, dA 
  - \int_{S(x)} r^{1 - 2m} \frac{\partial v}{\partial r} v e^{-cr+w} \, dA \\ 
& - \int_{E(x, t)} r^{- 2m} \frac{\partial v}{\partial r} v e^{-cr+w} \, dv_g. 
\end{align*}
The assumption $(22)$ implies that $\displaystyle \liminf_{t\to \infty} \int_{S(t)} r^{1 - 2m} \frac{\partial v}{\partial r} v e^{-cr+w} \, dA =0$, and hence, by substituting an appropriate divergence sequence $\{ t_j \}$ into the equation above, we have
\begin{align}
& \int_{E(x, \infty)} r^{1 - 2m} \Big\{ |\nabla v|^2 
  - \Big( \lambda + q_1 - \big( \rho '(r) + \frac{c}{2} \big)q_{\star} \Big) v^2  \Big\} e^{-cr+w} \, dv_g \notag \\ 
= 
& - \int_{S(x)} r^{1 - 2m} \frac{\partial v}{\partial r} v e^{-cr+w} \, dA  - \int_{E(x, \infty)} r^{- 2m} \frac{\partial v}{\partial r} v e^{-cr+w} \, dv_g. 
\end{align}
Lemma $4.1$ immediately follows from $(19)$ and $(20)$. 
\end{proof}
%%
%%%%%     Lemma 4.2     %%%%%
%%
\begin{lem}
For any $k \in \mathbb{R}$ and $r_0 \le s < t$, we have 
\begin{align*}
& \int_{S(t)} e^{kr} u^2 e^{-cr+w} \,dA 
  - \int_{S(s)} e^{kr} u^2 e^{-cr+w} \,dA \\
=
& \int_{E(s, t)} e^{kr} \big\{ k + q_{\star} \big\} u^2 e^{-cr+w} \,dv_g 
  + 2\int_{E(s, t)} e^{kr} \frac{\partial u}{\partial r} u e^{-cr+w} \,dv_g. 
\end{align*}
\end{lem}
\begin{proof}
A direct computation shows that
\begin{align*}
  {\rm div}(e^{kr} u^2 e^{-cr+w} \nabla r ) 
= e^{kr} \bigg\{ \Big( k+ \Delta_g r + \frac{\partial w}{\partial r} - c \Big) 
  u^2 + 2 \frac{\partial u}{\partial r} u \bigg\} e^{-cr+w} . 
\end{align*}
In view of $(2)$, integration of this equation over $E(s, t)$ with respect to $v_g$ yields Lemma $4.2$. 
\end{proof}
%%
%%%%%     Theorem 4.1     %%%%%
%%
\begin{thm}
Under the assumptions of Theorem $3.1$, we have, for any $ k > 0 $, 
\begin{align}
  \int_{E(r_0, \infty)} e^{kr} \left\{ u^2 + |\nabla u|^2 \right\} 
  e^{-cr+w} \,dv_g < \infty.
\end{align}
\end{thm}
\begin{proof}
Let $k_0$ be a ``nonnegative'' constant. 
In order to prove Theorem $4.1$, we shall assume that 
\begin{equation}
  \int_{E(r_0, \infty)} r^{m} e^{2 k_0 r} \big\{ u^2 + |\nabla u|^2 \big\} 
  e^{-cr+w} \,dv_g < \infty  \qquad {\rm for~all}~~m \ge 1
\end{equation}
and show that, there exist positive constants $\overline{c}_4 = \overline{c}_4(A_1, \varepsilon)$, $\overline{c}_5 = \overline{c}_5(A_1, \varepsilon)$, and $\overline{c}_6 = \overline{c}_6(A_1, \varepsilon)$, independent of $k_0 \ge 0$, such that 
\begin{align}
& \int_{E(r_0, \infty)} e^{ 2(k_0 + k) r} u^2 e^{-cr+w} \, dv_g 
  < \infty \notag \\ 
& \hspace{30mm} {\rm for~any}~~0< k < \sqrt{\big\{ (\overline{c}_4)^2 + \overline{c}_5 \big\}(k_0)^2 + \overline{c}_6} - \overline{c}_4 k_0 , 
\end{align}
where $\varepsilon$ is a fixed constant satisfying 
\begin{equation}
 A_1 < \varepsilon < 2 \widetilde{\alpha}_1 - B_1 . 
\end{equation}
For that purpose, we shall set 
\begin{align}
  \rho(r) = k_0 r + m \log r \quad {\rm and} \quad \gamma = \varepsilon - b
\end{align} 
in Proposition $2.3$. 
Then, we have
\begin{align}
& v = r^m e^{k_0 r} u~; 
  \quad q_1 = - \rho''(r) +(\rho'(r))^2 
  = (k_0)^2 + \frac{m^2 + m}{r^2} + 2 k_0 \frac{m}{r} ~; \\
& r \frac{\partial q_1}{\partial r} = -2 \frac{m^2 + m}{r^2} 
  - 2 k_0 \frac{m}{r} \notag . 
\end{align}
For convenience, we shall set 
$$
  b_{\max} 
  := \max \big\{ |b-A_1|, |b+B_1| \big\}~;~{\rm then},~|r q_{\star}| \le b_{\max} .
$$
Hence, we have, for $r_0 \le s < t$, 
\begin{align}
& \int_{S(t)} r^{\varepsilon-b} \bigg\{ 
  \bigg( \frac{\partial v}{\partial r} \bigg)^2
  + \frac{1}{2} (\lambda + q_1) v^2 - \frac{1}{2} |\nabla v|^2 \bigg\} 
  e^{-cr+w} \, dA \notag \\
& + \frac{1}{2} \int_{S(s)} r^{\varepsilon-b} 
  \bigg\{ |\nabla v|^2 - (\lambda + q_1) v^2 
  + \Big( \rho'(r) + \frac{c}{2} \Big)q_{\star} v^2 \notag \\
& \hspace{40mm} + b_{\max} \Big( \frac{m}{r^2} 
  + \frac{|c + 2k_0|}{2r} \Big) v^2 
  - 2 \bigg( \frac{\partial v}{\partial r} \bigg)^2 \bigg\} 
  e^{-cr+w} \, dA \notag \\
\ge 
& \int_{S(t)} r^{\varepsilon-b} \bigg\{ 
  \bigg( \frac{\partial v}{\partial r} \bigg)^2
  + \frac{1}{2} (\lambda + q_1) v^2 - \frac{1}{2} |\nabla v|^2 \bigg\} 
  e^{-cr+w} \, dA \notag \\
& + \frac{1}{2} \int_{S(s)} r^{\varepsilon-b} 
  \bigg\{ |\nabla v|^2 - (\lambda + q_1) v^2 
  + \Big( \rho'(r) + \frac{c}{2} \Big)q_{\star} v^2 \notag \\
& \hspace{50mm} - \Big( \rho'(r) + \frac{c}{2} \Big)q_{\star} v^2 
  - 2 \bigg( \frac{\partial v}{\partial r} \bigg)^2 \bigg\} 
  e^{-cr+w} \, dA \notag \\
\ge 
& \int_{E(s, t)} r^{\varepsilon-b-1} 
  \bigg\{ 2 k_0 r + 2m + \frac{\varepsilon-B_1-2b}{2} \bigg\}
  \bigg( \frac{\partial v}{\partial r} \bigg)^2 e^{-cr+w} \, dv_g \notag \\
& + \int_{E(s, t)} r^{\varepsilon-b-1} 
  \bigg\{ k_0 + \frac{m}{r} + \frac{c}{2} \bigg\} rq_{\star} 
  \frac{\partial v}{\partial r} v e^{-cr+w} \, dv_g \notag \\
& + \frac{1}{2} \int_{E(s, t)} r^{\varepsilon-b-1} 
  \bigg\{ (\lambda + k_0^2)(\varepsilon - A_1) 
  - \frac{m^2+m}{r^2}(2 + A_1 - \varepsilon) \notag \\ 
& \hspace{50mm}  - 2k_0 \frac{m}{r}(1 + A_1 - \varepsilon) \bigg\} v^2 
  e^{-cr+w} \, dv_g , 
\end{align}
where we have used the facts, 
\begin{align*}
  -A_1 \le rq_{\star} - b \le B_1 ~~; \quad 
  r(\nabla dr)(\nabla v, \nabla v) 
  - \frac{1}{2} (\varepsilon + B_1) \, \widetilde{g} (\nabla v, \nabla v) 
  \ge 0 . 
\end{align*}
Now, substituting $\beta = \varepsilon-b-2$ and $\beta = \varepsilon-b-1$ into the equation in Lemma $2.2$ and multiplying them by $\frac{mb_{\max}}{2}$ and $\frac{|c+2k_0|b_{\max}}{4}$ respectively, we have
\begin{align}
& \frac{mb_{\max}}{2} \int_{S(t)} r^{\varepsilon-b-2} v^2 e^{-cr+w} \, dA 
  - \frac{mb_{\max}}{2} \int_{S(s)} r^{\varepsilon-b-2} v^2 e^{-cr+w} 
  \, dA \notag \\
= 
& \frac{mb_{\max}}{2} \int_{E(s, t)} r^{\varepsilon-b-2} 
  \bigg\{ \Big( q_{\star} + \frac{\varepsilon-b-2}{r} \Big) v^2 + 2 v \frac{\partial v}{\partial r} \bigg\} e^{-cr+w} \, dv_g
\end{align}
and 
\begin{align}
& \frac{|c+2k_0|b_{\max}}{4} \int_{S(t)} r^{\varepsilon-b-1} v^2 e^{-cr+w}\,dA 
  - \frac{|c+2k_0|b_{\max}}{4} \int_{S(s)} r^{\varepsilon-b-1} v^2 e^{-cr+w} 
  \, dA \notag \\
= 
& \frac{|c+2k_0|b_{\max}}{4} \int_{E(s, t)} r^{\varepsilon-b-1} 
  \bigg\{ \Big( q_{\star} + \frac{\varepsilon-b-1}{r} \Big) v^2 + 2 v \frac{\partial v}{\partial r} \bigg\} e^{-cr+w} \, dv_g
\end{align}
Thus, combining $(27)$, $(28)$, and $(29)$, we obtain 
\begin{align}
& \frac{1}{2} \int_{S(t)} r^{\varepsilon-b} \bigg\{ 
  2 \bigg( \frac{\partial v}{\partial r} \bigg)^2 
  + (\lambda + q_1) v^2 - |\nabla v|^2 + \frac{mb_{\max}}{r^2} v^2 
  + \frac{|c+2k_0|b_{\max}}{2r} v^2 \bigg\} e^{-cr+w} \, dA \notag \\
& + \frac{1}{2} \int_{S(s)} r^{\varepsilon-b} 
  \bigg\{ |\nabla v|^2 - (\lambda + q_1) v^2 
  + \Big( \rho'(r) + \frac{c}{2} \Big) q_{\star} v^2 
  - 2 \bigg( \frac{\partial v}{\partial r} \bigg)^2 \bigg\} 
  e^{-cr+w} \, dA \notag \\
\ge 
& \int_{E(s, t)} r^{\varepsilon-b-1} 
  \bigg\{ 2 k_0 r + 2m + \frac{\varepsilon-B_1-2b}{2} \bigg\} 
  \bigg( \frac{\partial v}{\partial r} \bigg)^2 e^{-cr+w} \, dv_g \notag \\
& + \int_{E(s, t)} r^{\varepsilon-b-1} 
  \Big\{ \frac{m}{r}(b_{\max} + rq_{\star}) 
  + \frac{|2k_0+c|b_{\max} + (2k_0+c)rq_{\star}}{2} \Big\} 
  \frac{\partial v}{\partial r} v e^{-cr+w} \, dv_g \notag \\
& + \frac{1}{2} \int_{E(s, t)} r^{\varepsilon-b-1} 
  \Bigg\{ \big(\lambda + (k_0)^2 \big)(\varepsilon - A_1) 
  - \frac{m^2}{r^2} \Big(1 + \frac{1}{m}\Big) (2 + A_1 - \varepsilon) 
  \notag \\
& \hspace{40mm} 
  - \frac{m}{r} 2k_0 (1+A_1-\varepsilon) 
  \notag \\
& \hspace{50mm} 
  + b_{\max} \Big( q_{\star} + \frac{\varepsilon-b-2}{r} \Big) \Big(\frac{m}{r}+ \frac{|2k_0+c|}{2}\Big) \Bigg\} v^2 e^{-cr+w} \, dv_g . 
\end{align}
Here, we have
\begin{align}
  \left| \frac{m}{r}(b_{\max} + rq_{\star}) + \frac{|2k_0+c|b_{\max} 
  + (2k_0+c)rq_{\star}}{2} \right|
  \le b_{\max} \Big\{ \frac{2m}{r} + |2k_0+c| \Big\}; 
\end{align}
moreover, since $q_{\star} \ge - \frac{b_{\max}}{r}$, we have
\begin{align}
& \big(\lambda + (k_0)^2 \big)(\varepsilon - A_1) 
  - \frac{m^2}{r^2} \Big(1 + \frac{1}{m}\Big) (2 + A_1 - \varepsilon) 
  - \frac{m}{r} 2k_0 (1+A_1-\varepsilon) \notag \\ 
& + b_{\max} \Big( q_{\star} + \frac{\varepsilon-b-2}{r} \Big) 
  \Big(\frac{m}{r}+ \frac{|2k_0+c|}{2}\Big) \notag \\
\ge 
& \big(\lambda + (k_0)^2 \big)(\varepsilon - A_1) 
  - \frac{m^2}{r^2} \bigg\{ \Big(1+\frac{1}{m}\Big)(2+A_1-\varepsilon) 
  + \frac{b_{\max}(b_{\max}+2+b-\varepsilon)}{m} \bigg\} \notag \\
& - \frac{m}{r} \bigg\{ 2k_0(1+A_1-\varepsilon) 
  + \frac{b_{\max}(b_{\max}+2+b-\varepsilon)|2k_0+c|}{m} \bigg\} \notag \\
= 
& \big(\lambda + (k_0)^2 \big)(\varepsilon - A_1) - \frac{m^2}{r^2} P_2
  - \frac{m}{r} P_1 , 
\end{align}
where we set 
\begin{align*}
& P_2:=P_2(m,A_1,B_1,b) 
  = \Big(1+\frac{1}{m}\Big)(2+A_1-\varepsilon) + \frac{b_{\max}(b_{\max}+2+b-\varepsilon)}{m} ; \\
& P_1:=P_1(k_0,A_1,B_1,c) 
  = 2k_0(1+A_1-\varepsilon) 
  + \frac{b_{\max}(b_{\max}+2+b-\varepsilon)|2k_0+c|}{m},
\end{align*}
for simplicity. 

Now, let $\alpha>0$ be a fixed constant, and we shall substitute $\beta = \varepsilon-b-1$ into the equation of Lemma $2.2$; then, we have
\begin{align}
&\int_{S(t)} r^{\varepsilon-b} \frac{\alpha}{r} v^2 e^{-cr+w} \, dA 
 - \int_{S(s)} r^{\varepsilon-b} \frac{\alpha}{r} v^2 e^{-cr+w} \, dA \notag \\
=
& \int_{E(s, t)} r^{\varepsilon-b-1} \bigg\{ 
  \frac{\alpha}{r} (rq_{\star}-b+ \varepsilon - 1) v^2 
  + 2 \alpha v \frac{\partial v}{\partial r} \bigg\} e^{-cr+w} \,dv_g \notag \\
\ge 
& \int_{E(s, t)} r^{\varepsilon-b-1} \bigg\{ 
  - \frac{\alpha}{r} ( 1 + A_1 - \varepsilon ) v^2 
  + 2 \alpha v \frac{\partial v}{\partial r} \bigg\} e^{-cr+w} \,dv_g .
\end{align}
Combining $(30)$, $(31)$, $(32)$, and $(33)$ makes 
\begin{align}
& \frac{1}{2} \int_{S(t)} r^{\varepsilon-b} \bigg\{ 
  2 \bigg( \frac{\partial v}{\partial r} \bigg)^2 
  + \bigg( \lambda + q_1 + \frac{mb_{\max}}{r^2} 
  + \frac{|c + 2 k_0| b_{\max} + 4 \alpha}{2r} \bigg) v^2 - |\nabla v|^2 \bigg\}  e^{-cr+w} \, dA \notag \\
& + \frac{1}{2} \int_{S(s)} r^{\varepsilon-b} 
  \bigg\{ |\nabla v|^2 - (\lambda + q_1) v^2 
  + \Big( \rho'(r) + \frac{c}{2} \Big) q_{\star} v^2 - \frac{2\alpha}{r}v^2 
  - 2 \bigg( \frac{\partial v}{\partial r} \bigg)^2 \bigg\} 
  e^{-cr+w} \, dA \notag \\
\ge 
& \int_{E(s, t)} r^{\varepsilon-b-1} 
  \bigg\{ 2 k_0 r + 2m + \frac{\varepsilon-B_1-2b}{2} \bigg\} 
  \bigg( \frac{\partial v}{\partial r} \bigg)^2 e^{-cr+w} \, dv_g \notag \\
& - \int_{E(s, t)} r^{\varepsilon-b-1} \bigg\{ 2 \alpha + b_{\max} \Big( \frac{2m}{r} + |2k_0+c| \Big) \bigg\} 
  \left| \frac{\partial v}{\partial r} v \right| e^{-cr+w} \, dv_g \notag \\
& + \frac{1}{2} \int_{E(s, t)} r^{\varepsilon-b-1} 
  \Bigg\{ \big(\lambda + (k_0)^2 \big)(\varepsilon - A_1) - \frac{m^2}{r^2} P_2
  - \frac{m}{r} \widetilde{P}_1 \Bigg\} v^2 e^{-cr+w} \, dv_g ,
\end{align}
where we set 
$$
  \widetilde{P}_1 := P_1 + \frac{2\alpha (1+A_1-\varepsilon)}{m}
$$
for simplicity. 
From $(22)$ and $(26)$, we see that
\begin{align*}
  \liminf_{t \to \infty} \int_{S(t)} r^{\varepsilon} \Bigg\{ 
  2 \bigg( \frac{\partial v}{\partial r} \bigg)^2 
  + \bigg( \lambda + q_1 + \frac{mA_1}{r^2} 
 + \frac{(c + 2 k_0) A_1 + 4 \alpha}{2r} \bigg) v^2 
 - |\nabla v|^2 \Bigg\} 
  e^{-cr+w} \, dA = 0. 
\end{align*}
Hence, taking an appropriate divergent sequence $\{ t_i \}$, putting $t = t_i$ in $(34)$ and letting $t_i \to \infty$, we obtain 
\begin{align}
& \frac{1}{2} \int_{S(s)} r^{\varepsilon-b} 
  \bigg\{ |\nabla v|^2 - (\lambda + q_1) v^2 
  + \Big( \rho'(r) + \frac{c}{2} \Big) q_{\star} v^2 - \frac{2\alpha}{r}v^2 
  - 2 \bigg( \frac{\partial v}{\partial r} \bigg)^2 \bigg\} 
  e^{-cr+w} \, dA \notag \\
\ge 
& \int_{E(s,\infty)} r^{\varepsilon-b-1} 
  \bigg\{ 2 k_0 r + 2m + \frac{\varepsilon-B_1-2b}{2} \bigg\} 
  \bigg( \frac{\partial v}{\partial r} \bigg)^2 e^{-cr+w} \, dv_g \notag \\
& - \int_{E(s,\infty)} r^{\varepsilon-b-1} 
  \bigg\{ 2 \alpha + b_{\max} \Big( \frac{2m}{r} + |2k_0+c| \Big) \bigg\} 
  \left| \frac{\partial v}{\partial r} v \right| e^{-cr+w} \, dv_g \notag \\
& + \frac{1}{2} \int_{E(s,\infty)} r^{\varepsilon-b-1} 
  \Bigg\{ \big(\lambda + (k_0)^2 \big)(\varepsilon - A_1) - \frac{m^2}{r^2} P_2
  - \frac{m}{r} \widetilde{P}_1 \Bigg\} v^2 e^{-cr+w} \, dv_g . 
\end{align}
Now, we shall set 
\begin{align*}
& C_2 := 2 k_0 r + 2m + \frac{\varepsilon-B_1-2b}{2} ~; \notag \\
& C_3 := 2 \alpha + b_{\max} \Big( \frac{2m}{r} + |2k_0+c| \Big) ~; \\
& C_4 := \big(\lambda + (k_0)^2 \big)(\varepsilon - A_1) - \frac{m^2}{r^2} P_2
  - \frac{m}{r} \widetilde{P}_1,
\end{align*}
and note that, in general, $a X^2 - b XY \ge - \frac{b^2}{4a} Y^2$ if $a > 0$; then, we have 
\begin{align}
  C_2 \bigg( \frac{\partial v}{\partial r} \bigg)^2 
  - C_3 \bigg| \frac{\partial v}{\partial r} v \bigg| 
  + \frac{C_4}{2} v^2 
\ge 
  \frac{1}{4} \bigg\{ 2C_4 -\frac{(C_3)^2}{C_2} \bigg\} v^2 .
\end{align}

In view of the definitions $P_1$, $P_2$, and $\widetilde{P_1}$, simple computation shows that, for any $0< \theta < 1$, there exist constants $m_0=m_0(A_1, B_1, b, c, k_0, \alpha, \theta)$ and $r_1=r_1(A_1, B_1, b, c, k_0, \alpha, \theta)$ such that, for any $m \ge m_0$ and $r \ge r_1$, the following inequality holds:
\begin{align}
  \frac{1}{4} \bigg\{ 2C_4 -\frac{(C_3)^2}{C_2} \bigg\} 
\ge 
  \frac{1}{2} \bigg\{ \overline{c}_1 
  - \Big(\frac{m}{r} \Big) 2k_0 \overline{c}_2 
  - \Big(\frac{m}{r} \Big)^2 \overline{c}_3 \bigg\}.
\end{align}
Here, we set
\begin{align*}
  \overline{c}_1 & = \overline{c}_1 (k_0) 
 := \big( \lambda + (k_0)^2 \big)(\varepsilon - A_1)(1 - \theta) ; \\
  \overline{c}_2 
& := \min \big\{ (1+A_1-\varepsilon)(1+\theta), \, \theta \big\} ; \\
  \overline{c}_3 
& := \min \big\{ (2+A_1-\varepsilon)(1+\theta), \, \theta \big\}, 
\end{align*}
because we do not know signs of constants $1+A_1-\varepsilon$ and $2+A_1-\varepsilon$. 
Note that constants, $\overline{c}_1$, $\overline{c}_2$, and $\overline{c}_3$, are positive. 
Thus, combining $(35)$, $(36)$, and $(37)$, we obtain, for $m \ge m_0$ and $r \ge r_1$, 
\begin{align*}
& s^{\varepsilon} \int_{S(s)} 
  \bigg\{ |\nabla v|^2 - (\lambda + q_1) v^2 
  + \Big( \rho'(r) + \frac{c}{2} \Big) q_{\star} v^2 \bigg\} 
  e^{-cr+w} \, dA \notag \\
& - s^{\varepsilon} \int_{S(s)} \bigg\{ \frac{2\alpha}{r}v^2 
  + 2 \bigg( \frac{\partial v}{\partial r} \bigg)^2 \bigg\} 
  e^{-cr+w} \, dA \notag \\
\ge 
& \int_{E(s, \infty)} r^{\varepsilon -1} 
  \bigg\{ \overline{c}_1 - \Big(\frac{m}{r} \Big) 2k_0 \overline{c}_2 
  - \Big(\frac{m}{r} \Big)^2 \overline{c}_3 \bigg\} v^2 e^{-cr+w} \, dv_g. 
\end{align*}
Multiplying both sides of the inequality above by $s^{1-2m-\varepsilon}$, and integrating it with respect to $s$ over $[x, \infty)$, we obtain, for $x \ge r_1$, 
\begin{align*}
& \int_{E(x , \infty)} r^{1-2m} \Big\{ |\nabla v|^2 - (\lambda + q_1)v^2 
  + \Big( \rho'(r) + \frac{c}{2} \Big) q_{\star} v^2 \Big\} 
  e^{-cr+w} \, dv_g \notag \\
& - \int_{E(x , \infty)} r^{1-2m} \bigg\{ \frac{2\alpha}{r} v^2 
  + 2 \Big( \frac{\partial v}{\partial r} \Big)^2 \bigg\} 
  e^{-cr+w} \, dv_g \notag \\
\ge 
& \int_{x}^{\infty} s^{1-2m-\varepsilon} \, ds \int_{E(s, \infty)} 
  r^{\varepsilon -1} 
  \bigg\{ \overline{c}_1 - \Big(\frac{m}{r} \Big) 2k_0 \overline{c}_2 
  - \Big(\frac{m}{r} \Big)^2 \overline{c}_3 \bigg\} v^2 e^{-cr+w} \, dv_g \\
\ge 
& \int_{x}^{\infty} s^{1-2m-\varepsilon} \bigg\{ \overline{c}_1 
  - \Big(\frac{m}{s} \Big) 2k_0 \overline{c}_2 
  - \Big(\frac{m}{s} \Big)^2 \overline{c}_3 \bigg\} \, ds \int_{E(s, \infty)} 
  r^{\varepsilon -1} v^2 e^{-cr+w} \, dv_g \\ 
\ge 
& \bigg\{ \overline{c}_1 
  - \Big(\frac{m}{x} \Big) 2k_0 \overline{c}_2 
  - \Big(\frac{m}{x} \Big)^2 \overline{c}_3 \bigg\} 
  \int_{x}^{\infty} s^{1-2m-\varepsilon} \, ds \int_{E(s, \infty)} 
  r^{\varepsilon -1} v^2 e^{-cr+w} \, dv_g .
\end{align*}
Substitution of the equation in Lemma $4.1$ into the inequality above yields 
\begin{align*}
& - \frac{1}{2} \frac{d}{dx}\bigg( x^{1-2m} \int_{S(x)} v^2 e^{-cr+w} 
  \, dA \bigg) 
  - \frac{1}{2} \int_{S(x)} r^{-2m} \big\{ 2m - 1 - rq_{\star} \big\} v^2 
  e^{-cr+w} \, dA \\
& - \int_{E(x, \infty)} r^{1-2m} \bigg\{ \frac{2\alpha}{r} v^2 
  + \frac{1}{r} \frac{\partial v}{\partial r} v 
  + 2 \Big( \frac{\partial v}{\partial r} \Big)^2 \bigg\} e^{-cr+w} \, dv_g \\
\ge 
& \bigg\{ \overline{c}_1 
  - \Big(\frac{m}{x} \Big) 2k_0 \overline{c}_2 
  - \Big(\frac{m}{x} \Big)^2 \overline{c}_3 \bigg\} 
  \int_{x}^{\infty} s^{1-2m-\varepsilon} \, ds \int_{E(s, \infty)} 
  r^{\varepsilon -1} v^2 e^{-cr+w} \, dv_g .
\end{align*}
Here, 
\begin{align*}
& \frac{2\alpha}{r} v^2 
  + \frac{1}{r} \frac{\partial v}{\partial r} v 
  + 2 \Big( \frac{\partial v}{\partial r} \Big)^2 
\ge 
  \frac{2\alpha}{r} \bigg\{ 1 - \frac{1}{16 \alpha r} \bigg\} v^2 
\ge 0, \qquad {\rm if}~ r \ge \frac{1}{16\alpha} ~; \\ 
& 2m -1 - r q_{\star} \ge 2m \Big( 1 - \frac{1+B_1}{2m} \Big) 
  \ge 2 (1 - \theta) m, \qquad {\rm if}~m \ge \frac{1+B_1}{2\theta}. 
\end{align*}
Therefore, we obtain, for any $x \ge r_2 := \max \{ r_1, \frac{1}{16 \alpha} \}$ and $m \ge m_1 := \max \{ m_0, \frac{1+B_1}{2\theta} \}$, 
\begin{align*}
& - \frac{1}{2} \frac{d}{dx}\bigg( x^{1-2m} \int_{S(x)} v^2 e^{-cr+w} 
  \, dA \bigg) 
  - (1 - \theta) \frac{m}{x} \bigg( x^{1-2m} \int_{S(x)} v^2 e^{-cr+w} 
  \, dA \bigg) \\
\ge 
& \bigg\{ \overline{c}_1 
  - \Big(\frac{m}{x} \Big) 2k_0 \overline{c}_2 
  - \Big(\frac{m}{x} \Big)^2 \overline{c}_3 \bigg\} 
  \int_{x}^{\infty} s^{1-2m-\varepsilon} \, ds \int_{E(s, \infty)} 
  r^{\varepsilon -1} v^2 e^{-cr+w} \, dv_g .
\end{align*}
For $x \ge r_2$ and $m \ge m_1$, we shall set
\begin{align*}
& \frac{m}{x} = \frac{-k_0\overline{c}_2 + \sqrt{(k_0)^2 (\overline{c}_2)^2 + \overline{c}_3 \big( \lambda + (k_0)^2 \big)(\varepsilon - A_1)(1 - \theta)}}{\overline{c}_3} =: \overline{c}_7(k_0) ~;\\ 
& F(x) := x^{1-2m} \int_{S(x)} v^2 e^{-cr+w} \, dA 
  = x \int_{S(x)} e^{2k_0r} u^2 e^{-cr+w} \, dA , 
\end{align*}
where we shall recall $\overline{c}_1=\big( \lambda + (k_0)^2 \big)(\varepsilon - A_1)(1 - \theta)$. 
Then, the inequality above reduced to 
\begin{align*}
  F'(x) + 2(1 - \theta) \overline{c}_7 F(x) \le 0 
  \qquad 
  {\rm for}~~x \ge r_3 := \max \Big\{ r_2, \frac{m_1}{\overline{c}_7} \Big\}. 
\end{align*}
Thus, $G(x):= e^{2(1 - \theta) \overline{c}_7 x} F(x)$ satisfies $G'(x)$ for $x \ge r_3$, and hence, $G(x) \le G(r_3)$ for $x \ge r_3$, that is, 
\begin{align*}
  x \int_{S(x)} e^{2k_0r} u^2 e^{-cr+w} \, dA = F(x) 
  \le e^{-2(1 - \theta) \overline{c}_7 x} G(r_3). 
\end{align*}
The desired assertion $(23)$ follows from this inequality and the definition of $\overline{c}_7 = \overline{c}_7(k_0)$ above, where we shall recall that $\overline{c}_2$ and $\overline{c}_3$ are independent of $k_0$. 

Now, we shall consider an increasing sequence $\{ a_n \}_{n=0}^{\infty}$ of nonnegative numbers defined by 
\begin{align*}
 a_{n+1} = a_n + \sqrt{ \big\{ (\overline{c}_4)^2 + \overline{c}_5 \big\} (a_n)^2 + \overline{c}_6 } - \overline{c}_4 a_n, \quad a_0 = 0. 
\end{align*}
Then, $\lim_{n \to \infty} a_n = \infty$. 
Indeed, if contrary, there exists $a_{\infty}:= \lim_{n\to \infty} a_n \in (0, \infty)$. 
Taking the limit, we have  $a_{\infty} = a_{\infty} + \sqrt{ \{ (\overline{c}_4)^2 + \overline{c}_5 \} (a_{\infty})^2 + \overline{c}_6 } - \overline{c}_4 a_{\infty}$, and hence, $\sqrt{ \{ (\overline{c}_4)^2 + \overline{c}_5 \} (a_{\infty})^2 + \overline{c}_6 } = \overline{c}_4 a_{\infty}$; this contradicts the facts: $\overline{c}_5 > 0$ and $\overline{c}_6 > 0$. 
Therefore, by virtue of $(23)$ combined with $\lim_{n \to \infty} a_n = \infty$, we obtain 
\begin{align}
  \int_{E(r_0, \infty)} e^{ k r} u^2 e^{-cr+w} \,dv_g < \infty 
 \quad {\rm for~any}~~0< k < \infty. 
\end{align}
Next, we shall show that, $(38)$ implies that
\begin{align}
  \int_{E(r_0, \infty)} e^{ k r} |\nabla u|^2 e^{-cr+w} \,dv_g < \infty 
 \quad {\rm for~any}~~0< k < \infty. 
\end{align}
Since $(38)$ implies that
\begin{align*}
 \liminf_{t \to \infty} \int_{S(t)} e^{ k r} u^2 e^{-cr+w} \,dA = 0, 
\end{align*}
taking an appropriate divergent sequence $\{ t_i \}$, and letting $t = t_i \to \infty$ in the equation of Lemma $4.2$, we obtain 
\begin{align*}
& 2\int_{E(s, \infty)} e^{kr} \frac{\partial u}{\partial r} u 
  e^{-cr+w} \,dv_g \\
= 
& - \int_{S(s)} e^{kr} u^2 e^{-cr+w} \,dA 
  - \int_{E(s, \infty)} e^{kr} \big\{ k + q_{\star} \big\} u^2 
  e^{-cr+w} \,dv_g , 
\end{align*}
where the right hand side of this equation is finite by $(38)$. 
In particular, we have 
\begin{align}
 \liminf_{R \to \infty} e^{ k R} 
 \bigg| \int_{S(R)} \frac{\partial u}{\partial r} u e^{-cr+w} \,dA \bigg| = 0. 
\end{align}
Now, we shall put $\rho=0$ and $\psi = e^{kr}$ in Proposition $2.1$; then, $v=u$ and $q_1=0$, and hence, we have 
\begin{align*}
& \int_{E(s, t)} e^{kr} \Big\{ |\nabla u|^2 
  - \Big( \lambda - \frac{c}{2}q_{\star} \Big) u^2 \Big\} e^{-cr+w} \,dv_g \\ 
=
& \int_{S(t)} e^{kr} \frac{\partial u}{\partial r} u e^{-cr+w} \,dA 
  - \int_{S(s)} e^{kr} \frac{\partial u}{\partial r} u e^{-cr+w} \,dA 
  - k \int_{E(s, t)} e^{kr} \frac{\partial u}{\partial r} u e^{-cr+w} 
  \, dv_g \\ 
\le  
& \int_{S(t)} e^{kr} \frac{\partial u}{\partial r} u e^{-cr+w} \,dA 
  - \int_{S(s)} e^{kr} \frac{\partial u}{\partial r} u e^{-cr+w} \,dA 
  + \frac{k^2}{2} \int_{E(s, t)} e^{kr} u^2 e^{-cr+w} \,dv_g \\
& + \frac{1}{2} \int_{E(s, t)} e^{kr} |\nabla u|^2 e^{-cr+w} \,dv_g .
\end{align*}
Therefore, we have 
\begin{align*}
& \frac{1}{2} \int_{E(s, t)} e^{kr} |\nabla u|^2 e^{-cr+w} \,dv_g \\ 
\le  
& \int_{S(t)} e^{kr} \frac{\partial u}{\partial r} u e^{-cr+w} \,dA 
  - \int_{S(s)} e^{kr} \frac{\partial u}{\partial r} u e^{-cr+w} \,dA \\ 
& \hspace{30mm} + \int_{E(s, t)} e^{kr} \Big( \lambda - \frac{c}{2} q_{\star} 
  + \frac{k^2}{2} \Big) u^2 e^{-cr+w} \,dv_g .
\end{align*}
In view of $(40)$, by taking appropriate divergent sequence $\{t_i\}$, substituting it $t=t_i$ into the inequality above, and letting $t_i \to \infty$, we obtain \begin{align*}
& \frac{1}{2} \int_{E(s, \infty)} e^{kr} |\nabla u|^2 e^{-cr+w} \,dv_g \\ 
\le  
& - \int_{S(s)} e^{kr} \frac{\partial u}{\partial r} u e^{-cr+w} \,dA 
  + \int_{E(s, \infty)} e^{kr} 
  \Big( \lambda - \frac{c}{2} q_{\star} + \frac{k^2}{2} \Big) u^2 
  e^{-cr+w} \,dv_g .
\end{align*}
Since the right hand side of this inequality is finite by $(38)$, we obtain $(39)$. 
Thus, we have proved Theorem $4.1$. 
\end{proof}
%%
%%%%%     Section 5     %%%%%
%%
%%
\section{Vanishing}
%%
%%%%%     Theorem 5.1     %%%%%
%%
\begin{thm}
Under the assumption of Theorem $4.1$, 
\begin{align*}
  u \equiv 0 \qquad {\rm on}~~E(r_0,\infty).
\end{align*}
\end{thm}
\begin{proof}
Let $k\ge 1$ be a fixed constant, and take $\varepsilon$ so that
\begin{align}
  2 \widetilde{\alpha}_1 - B_1 > \varepsilon >A_1. 
\end{align}
We shall set $\rho (r) = kr$ and $\gamma = \varepsilon - b$ in Proposition $2.3$. 
Then, we have
\begin{align}
& v = e^{kr} u ~; \quad q_1 = k^2 ~; \\
& 2 r(\nabla dr)(\nabla v, \nabla v) 
  - (\varepsilon + rq_{\star}-b) \widetilde{g}(\nabla v, \nabla v) 
\ge 
  \big( 2 \widetilde{\alpha}_1 - \varepsilon - B_1 \big) 
  \widetilde{g}(\nabla v, \nabla v) \ge 0, \notag
\end{align}
and hence, 
\begin{align}
& \int_{S(t)} r^{\varepsilon-b} \bigg\{ 
  2 \bigg( \frac{\partial v}{\partial r} \bigg)^2 
  + \big( \lambda + k^2 \big)v^2 - |\nabla v|^2 \bigg\} 
  e^{-cr+w} \, dA \notag \\ 
& + \int_{S(s)} r^{\varepsilon-b} \bigg\{ |\nabla v|^2 
  - (\lambda + k^2) v^2 
  - 2 \bigg( \frac{\partial v}{\partial r} \bigg)^2 \bigg\} 
  e^{-cr+w} \, dA \notag \\ 
\ge 
& \int_{E(s, t)} r^{\varepsilon-b-1} \big\{ 4kr + \varepsilon - b - rq_{\star} 
  \big\} \bigg( \frac{\partial v}{\partial r} \bigg)^2 
  e^{-cr+w} \, dv_g \notag \\
& + \int_{E(s, t)} r^{\varepsilon-b-1} (2k + c) rq_{\star} 
  \frac{\partial v}{\partial r} v e^{-cr+w} \, dv_g \notag \\
& + \int_{E(s, t)} r^{\varepsilon-b-1} (\lambda + k^2)
  (\varepsilon + rq_{\star}-b)   v^2 e^{-cr+w} \, dv_g \notag \\ 
\ge 
& \int_{E(s, t)} r^{\varepsilon-b-1} 
  \big\{ 4kr + \varepsilon - 2b + B_1 \big\} 
  \bigg( \frac{\partial v}{\partial r} \bigg)^2 
  e^{-cr+w} \, dv_g \notag \\
& - \int_{E(s, t)} r^{\varepsilon-b-1} |2k + c| b_{\max} 
  \bigg| \frac{\partial v}{\partial r} v \bigg| e^{-cr+w} \, dv_g \notag \\
& + \int_{E(s, t)} r^{\varepsilon-b-1} (\lambda + k^2)(\varepsilon - A_1) v^2 
  e^{-cr+w} \, dv_g .
\end{align}
Now, in general, when $a>0$, $a X^2 - b XY \ge - \frac{b^2}{4a} Y^2$. 
Hence, 
\begin{align}
& \big\{ 4kr + \varepsilon - 2b + B_1 \big\} 
  \bigg( \frac{\partial v}{\partial r} \bigg)^2 
  - |2k + c| b_{\max} 
  \bigg| \frac{\partial v}{\partial r} v \bigg|
  + (\lambda + k^2)(\varepsilon - A_1) v^2 \notag \\ 
\ge 
& \left\{ (\lambda + k^2)(\varepsilon - A_1) 
  - \frac{(2k + c)^2 (b_{\max})^2}{4(4kr + \varepsilon - 2b + B_1)} \right\}   v^2 \notag \\ 
=
& \left\{ \lambda (\varepsilon - A_1) 
  + k \left( k(\varepsilon - A_1) - \frac{(2+\frac{c}{k})^2 (b_{\max})^2}{4(4r + \frac{\varepsilon - 2b + B_1}{k})} \right)
  \right\} v^2 .
\end{align}
Since $\varepsilon - A_1 >0$, there exist positive constant $k_1 = k_1( A_1, B_1, \varepsilon, b, c )$ such that the right hand side of $(44)$ is nonnegative for $k \ge k_1$ and $ r \ge \max\{ r_0, 1 \}$. 
Therefore, combining $(43)$ and $(44)$, we obtain, for $k \ge k_1$ and $t > s \ge \max\{ r_0, 1 \}$, 
\begin{align}
& \int_{S(t)} r^{\varepsilon-b} \bigg\{ 
  2 \bigg( \frac{\partial v}{\partial r} \bigg)^2 
  + \big( \lambda + k^2 \big)v^2 - |\nabla v|^2 \bigg\} 
  e^{-cr+w} \, dA \notag \\ 
& + \int_{S(s)} r^{\varepsilon-b} \bigg\{ |\nabla v|^2 
  - (\lambda + k^2) v^2 
  - 2 \bigg( \frac{\partial v}{\partial r} \bigg)^2 \bigg\} 
  e^{-cr+w} \, dA \ge 0 .
\end{align}
Here, in view of $(21)$ and $(42)$, we have
\begin{align*}
  \liminf _{t\to \infty} \int_{S(t)} r^{\varepsilon-b} \bigg\{ 
  2 \bigg( \frac{\partial v}{\partial r} \bigg)^2 
  + \big( \lambda + k^2 \big)v^2 - |\nabla v|^2 \bigg\} 
  e^{-cr+w} \, dA = 0.
\end{align*}
Hence, taking an appropriate divergent sequence $\{ t_i \}$ and letting $t = t_i \to \infty$ in $(45)$, we obtain, for any $k\ge k_1$ and $s \ge \max\{ r_0, 1 \}$,
\begin{align*}
  \int_{S(s)} \bigg\{ |\nabla v|^2 
  - 2 \bigg( \frac{\partial v}{\partial r} \bigg)^2 \bigg\} 
  e^{-cr+w} \, dA \ge 0 .
\end{align*}
Since $v = e^{kr} u $, we have 
\begin{align*}
  |\nabla v|^2 - 2 \bigg( \frac{\partial v}{\partial r} \bigg)^2 
= e^{2kr} \bigg\{ - k^2 u^2 - 2k \frac{\partial u}{\partial r} u 
  + |\nabla u|^2 - 2 \bigg( \frac{\partial u}{\partial r} \bigg)^2 \bigg\} .
\end{align*}
Therefore, we obtain, for any $k\ge k_1$ and $s \ge r_1 := \max\{ r_0, 1 \}$, 
\begin{align}
  - k^2 I_1(s) - k I_2(s) + I_3(s) \ge 0,
\end{align}
where 
\begin{align*}
& I_1(s) := \int_{S(s)} u^2 e^{-cr+w} \, dA ~; \quad 
  I_2(s) := 2 \int_{S(s)} \frac{\partial u}{\partial r} u e^{-cr+w} \, dA~; \\ 
& I_3(s) := \int_{S(s)} \bigg\{ |\nabla u|^2 - 2 \bigg( \frac{\partial u}{\partial r} \bigg)^2 \bigg\} e^{-cr+w} \, dA .
\end{align*}
Thus, for any fixed $s \ge r_1$, letting $k \to \infty$ in $(46)$, we obtain $I_1(s)=0$, that is, $u\equiv 0$ on $E( r_1, \infty )$. 
The unique continuation theorem implies that $u\equiv 0$ on $E$. 
\end{proof}

%%
%%%%%     section 6     %%%%%
%%
%%
\section{Radiation condition and growth property}
In this section, we shall briefly explain the relationship between the radiation conditions and the growth property $(*_3)$. 
In order to prove the limiting absorption principle in the author's paper \cite{K4}, it is an important step to show $u \equiv 0$ under the assumption $(*_3)$ (see Lemma $8.1$ in \cite{K4}). 

First, we shall introduced some terminology: for $s \in \mathbb{R}$, let $L^2_s(E, v_g)$ denote the space of all complex-valued measurable functions $f$ such that $\left|(1+r)^s f\right|$ is square integrable on $E$ with respect to $v_g$, and set 
\begin{align*}
  \| f \|_{L^2_s (E, v_g)} := \int_E \left( 1 + r \right)^{2s} |f|^2 \, dv_g .
\end{align*}
We also denote ${\it \Pi}_{+} := \{ x + iy \in \mathbb{C} \mid x > 0, \, y \ge 0 \}$ and ${\it \Pi}_{-} := \{ x + iy \in \mathbb{C} \mid x > 0, \, y \le 0 \}$. 

In \cite{K4}, the author studied Riemannian manifolds $(M,g)$ having ends $E_1, E_2, \cdots , E_m$ with radial coordinates, each of which satisfies either (I) or (II) below: 
\begin{align*}
& {\rm (I)} \hspace{6mm}
 \begin{cases}
  \ \displaystyle \nabla dr 
  \ge \Big\{ \frac{a_j}{r} + O (r^{-1-\delta}) \Big\} \, \widetilde{g} 
  \quad & {\rm on}~~E_j, \vspace{2mm} \\
  \ \displaystyle \Delta_g r = \frac{b_j}{r} + O (r^{-1-\delta}) 
  \quad & {\rm on}~~E_j; 
 \end{cases}
\\
& {\rm (II)} \hspace{5mm}
 \begin{cases}
  \ \displaystyle \nabla dr 
  \ge \Big\{ \frac{a_j}{r} + O (r^{-1-\delta}) \Big\} \, \widetilde{g} 
  \quad & {\rm on}~~E_j, \vspace{2mm} \\
  \ \displaystyle \Delta_g r = \beta_j + O (r^{-1-\delta}) 
  \quad & {\rm on}~~E_j, 
 \end{cases}
\end{align*}
where $a_j > 0$, $b_j > 0$, $\beta_j > 0$, and $\delta \in (0,1)$ are constants. 
For a solution $u$ of $- \Delta _g u - zu = f$ on $M$ and $f \in L^2_{\frac{1}{2}+s}(M, v_g)$, the author \cite{K4} introduced the radiation conditions as follows. 
For $E_j$ satisfying (I) and $z \in {\it \Pi}_{\pm}$,
\begin{align}
  u \in L^2_{- \frac{1}{2} - s'} \big( E_j, v_g \big)~;~
  \frac{\partial u}{\partial r} 
  + \left( \frac{b_j}{2r} \mp i \sqrt{z} \right) u \in 
  L^2_{-\frac{1}{2}+s} \big( E_j, v_g \big). 
\end{align}
For $E_j$ satisfying (II) and $z \in {\it \Pi}_{\pm}$ satisfying ${\rm Re}z > \frac{\beta^2}{4}$,
\begin{align}
  u \in L^2_{- \frac{1}{2} -s'} \big( E_j, v_g \big)~;~
  \frac{\partial u}{\partial r} 
  + \left( \frac{\beta_j}{2} \mp i \sqrt{z-\frac{(\beta_j)^2}{4}} \right)u \in 
  L^2_{-\frac{1}{2}+s} \big( E_j, v_g \big).
\end{align}
Here, $0 < s' < s < \min \{ \frac{1}{2}, a_{\min} \}$ are constants; $a_{\min}:= \min \{ a_j \mid 1 \le j \le m \}$; the square roots takes the principal value. 
(The condition $(48)$ above can be seen to be equivalent to $(14)$ in \cite{K4} by taking the multiplication operator $e^{\frac{\beta_j}{2}r}$ into account). 

Then, the following holds:
%%
%%%%%     PROPOSITION 6.1     %%%%%
%%
\begin{prop}
Let $u$ be a solution of $-\Delta_g u + \lambda u =0$ on an end $E$ with radial coordinates. 
Then, 
\begin{enumerate}[{\rm (1)}]
\item Assume that $u$ satisfies the radiation condition $(47)$ with $E_j = E$ and $z = \lambda > 0$. Then, $(*_3)$ with $\gamma = s -s'$ holds. Hence, if $E$ satisfies {\rm (I)} with $E_j = E$, then $u \equiv 0$ by Theorem $1.1$. 
\item Assume that $u$ satisfies the radiation condition $(48)$ with $E_j = E$ and $z = \lambda > \frac{\beta^2}{4}$. Then, $(*_3)$ with $\gamma = s -s'$ holds. Hence, if $E$ satisfies {\rm (II)} with $E_j = E$, then $u \equiv 0$ by Theorem $1.1$. 
\end{enumerate}
\end{prop}
\begin{proof}
We shall prove only (1), because the proof of (2) is quite similar. 
By considering the real and imaginary part of $u$, we assume that $u$ is real valued. 
For simplicity, we put $\rho_{\pm} := \frac{b}{2r} \mp i \sqrt{\lambda}$. 
Then, we have for $r_0 \le t$, 
\begin{align*}
  \mp \sqrt{\lambda} \int_{S(t)} u^2 \, dA 
= \int_{S(t)} u \big( {\rm Im} \, (\partial_r + \rho_{\pm})u \big) \, dA, 
\end{align*}
and hence, 
\begin{align*}
  \sqrt{\lambda} \int_{S(t)} u^2 \, dA 
\le 
  \int_{S(t)} |u| |(\partial_r + \rho_{\pm})u| \, dA, 
\end{align*}
where we write $(\partial_r + \rho_{\pm})u:= \frac{\partial u}{\partial r} + \rho_{\pm}u$ for simplicity. 
Multiplying the both sides of the inequality above by $(1+t)^{s-s'-1}$, and integrating it with respect to $t$ over $[r_0,\infty)$, we obtain
\begin{align*}
& \sqrt{\lambda} \int_{E(r_0, \infty)} (1+r)^{s-s'-1} u^2 \, dv_g \\ 
\le 
& \int_{E(r_0, \infty)} (1+r)^{s-s'-1} |u| |(\partial_r + \rho_{\pm})u| 
  \, dv_g \\ 
\le 
& \frac{1}{2} \int_{E(r_0, \infty)} (1+r)^{-1+2s} |(\partial_r + \rho_{\pm})u|^2  \, dv_g 
  + \frac{1}{2} \int_{E(r_0, \infty)} (1+r)^{-1-2s'} u^2 \, dv_g 
< \infty,
\end{align*}
where the right hand side of this inequality is finite by $(47)$. 
Hence, $-\Delta_g u = \lambda u \in L^2_{\frac{s-s'-1}{2}}(E, v_g)$, which implies that $|\nabla u| \in L^2_{\frac{s-s'-1}{2}}(E, v_g)$ as is shown below: we shall set $\ell := \frac{s-s'-1}{2}$, and define, for $t > r_0$ and $x \in E(r_0, \infty)$, 
\begin{align*}
 h_t (x) := 
 \begin{cases}
  \ \ \ \ \ 1 \qquad & \mbox{if}\quad r(x) \le t,\\
  - r(x) + t + 1 \qquad & \mbox{if}\quad t \le r(x) \le t+1,\\
  \ \ \ \ \ 0 \qquad & \mbox{if}\quad t + 1 \le r(x) .
 \end{cases}
\end{align*}
Then, by direct computations, we obtain, for $0 < \varepsilon <1$, 
\begin{align*}
& \int_{E(r_0, t+1)} (h_t)^2 (1+r)^{2\ell} |\nabla u|^2 \, dv_g \\
= 
& \int_{E(r_0, t+1)} \langle \nabla \{ (h_t)^2 (1+r)^{2\ell} u \}, 
  \nabla u \rangle \, dv_g \\ 
& - 2 \int_{E(r_0, t+1)} h_t (1+r)^{2\ell} \Big\{ h_t' + \frac{\ell}{1+r} \Big\}  u \frac{\partial u}{\partial r} \, dv_g \\ 
\le  
& - \int_{S(r_0)} (1+r)^{2\ell} u \frac{\partial u}{\partial r} dA 
  + \lambda \int_{E(r_0, t+1)} (h_t)^2 (1+r)^{2\ell} u^2 \, dv_g \\ 
& + \frac{(1 + |\ell|)^2}{\varepsilon} \int_{E(r_0, t+1)} (h_t)^2 (1+r)^{2\ell}
  u^2 \, dv_g 
  + \varepsilon \int_{E(r_0, t+1)} (h_t)^2 (1+r)^{2\ell} 
  \bigg( \frac{\partial u}{\partial r} \bigg)^2 \, dv_g, 
\end{align*}
and hence, 
\begin{align*}
& (1-\varepsilon) \int_{E(r_0, t+1)} (h_t)^2 (1+r)^{2\ell} |\nabla u|^2 
  \, dv_g \\
\le 
& - \int_{S(r_0)} (1+r)^{2\ell} u \frac{\partial u}{\partial r} dA 
  + \bigg\{ \frac{(1 + |\ell|)^2}{\varepsilon} + \lambda \bigg\} 
  \int_{E(r_0, t+1)} (h_t)^2 (1+r)^{2\ell} u^2 \, dv_g .
\end{align*}
Therefore, letting $t \to \infty$, we obtain 
\begin{align*}
& (1-\varepsilon) \int_{E(r_0, \infty)} (h_t)^2 (1+r)^{2\ell} |\nabla u|^2 
  \, dv_g \\
\le 
& - \int_{S(r_0)} (1+r)^{2\ell} u \frac{\partial u}{\partial r} dA 
  + \bigg\{ \frac{(1 + |\ell|)^2}{\varepsilon} + \lambda \bigg\} 
  \int_{E(r_0, \infty)} (h_t)^2 (1+r)^{2\ell} u^2 \, dv_g < \infty. 
\end{align*}
Thus, $u$, $|\nabla u| \in L^2_{\frac{s-s'-1}{2}}(E, v_g)$. 
Hence, $(*_3)$ with $\gamma = s -s'$ holds. 
\end{proof}

Proposition $6.1$ combined with the use of Lemma $2.1$ with $w=0$ implies that Lemma $8.1$ in \cite{K4} holds.

%%
%%%%%     SECTION 7     %%%%%
%%
%%
\section{Absolute continuity and complexity of metric at infinity}
In this section, we shall consider several Riemannian manifolds whose Laplacians are absolutely continuous, but the growth orders of their metrics on ends are complicated at infinity so that radial curvatures on ends diverge at infinity. 
\vspace{3mm}

%%
%%%%%     Metrics on $\mathbb{R}^2$     %%%%%
%%
\noindent{\bf Rotationally symmetric metrics on $\mathbb{R}^2$.} 
Let $(\mathbb{R}^2, g_f := dr^2 + f(r)^2 g_{S^1(1)})$ be a rotationally symmetric manifold, where $r$ stands for the Euclidean distance to the origin; $g_{S^1(1)}$ is the standard metric on $S^1(1) = \{ z \in \mathbb{C} \mid |z|=1 \}$. 
Let $b>0$, $\delta >0$, and $m >0$ be any constants and $r_0 \gg 1$ be a large constant. 
\vspace{2mm}

\noindent (i) For example, assume that 
\begin{align*}
 f (r) = \exp \left( \int_{r_0}^r \bigg\{ \frac{b}{t} + \frac{\sin \big( \exp (t^m) \big)}{t^{1 + \delta}} \bigg\} \, dt \right) \quad {\rm for}~~r \ge r_0. 
\end{align*}
Then, $\displaystyle \Delta_{g_f} r = \frac{b}{r} + O(r^{-1-\delta})$, and hence, the limiting absorption principle holds on respectively $\{ x + iy \mid x > 0,\,y \ge 0 \}$ and $\{ x + iy \mid x > 0,\,y \le 0 \}$, and hence $\Delta_{g_f}$  is absolute continuous on $(0,\infty)$ by Theorem $1.1$ and Theorem $1.2$ in \cite{K4}; in particular, $\sigma_{\rm pp}(-\Delta_{g_f}) = \emptyset$; however, the Gaussian curvature $K(r) = - \frac{f''}{f}(r)$ diverges, while oscillating, as $r \to \infty$. 
\vspace{2mm}

\noindent (ii) For example, assume that 
\begin{align*}
 f (r) = \exp \left( \int_{r_0}^r \bigg\{ b + \frac{\sin \big( \exp (t^m) \big)}{t^{1 + \delta}} \bigg\} \, dt \right) \quad {\rm for}~~r \ge r_0. 
\end{align*}
Then, $\displaystyle \Delta_{g_f} r = b + O(r^{-1-\delta})$, and hence, the limiting absorption principle holds on respectively $\{ x + iy \mid x > \frac{b^2}{4},\,y \ge 0 \}$ and $\{ x + iy \mid x > \frac{b^2}{4},\,y \le 0 \}$, and hence $\Delta_{g_f}$  is absolute continuous on $(\frac{b^2}{4}, \infty)$ by Theorem $1.1$ and Theorem $1.2$ in \cite{K4}; in particular, $\sigma_{\rm pp}(-\Delta_{g_f}) \cap (\frac{b^2}{4}, \infty) = \emptyset$; however, the Gaussian curvature $K(r) = - \frac{f''}{f}(r)$ diverges, while oscillating, as $r \to \infty$.

\vspace{3mm}

%%
%%%%%     Metrics on $[1, \infty) \times T^2}$     %%%%%
%%
\noindent{\bf Metrics on $\boldsymbol{[1, \infty) \times T^2}$.} 
Let $r$ be the standard coordinate on $[1, \infty)$ and $dr^2$ be the standard metric on $[1, \infty)$. 
\vspace{2mm}

\noindent (iii) Let $0 \le \varepsilon_0 \ll 1$ be any small constant, and $\{ \Phi_1, \Phi_2 \}$ be any $C^{\infty}$-partition of unity on $[1, \infty)$ satisfying $\displaystyle\Phi_1 (x), \Phi_2 (x) \ge \frac{\varepsilon_0}{x}$ for $x \in [1,\infty)$. 
Let $c > 0$ be any constant, and set 
\begin{align*}
 f_1 (r) := \exp \left( c \int_1^r \Phi_1 (t) \, dt \right) ~;~
 h_1 (r) := \exp \left( c \int_1^r \Phi_2 (t) \, dt \right) .
\end{align*}
We shall define a metric $g_1 = g_1(c)$ on $E := [1, \infty) \times S^1(1) \times S^1(1)$ by 
\begin{align*}
  g_1 = g_1(c_1) := dr^2 + f_1(r)^2 g_{S^1(1)} + h_1(r)^2 g_{S^1(1)} .
\end{align*}
Then, $\Delta_{g_1} r \equiv c$ and $\displaystyle \nabla dr \ge \frac{c \varepsilon_0}{r} \,\widetilde{g}_1$ on $(E, g_1(c))$. 
\vspace{3mm}

\noindent (iv) Let $0 \le \varepsilon_0 \ll \frac{1}{2}$ be any small constant, and $\{ \Phi_3, \Phi_4 \}$ be any $C^{\infty}$-partition of unity on $[1, \infty)$ satisfying $\Phi_3 (x), \Phi_4 (x) \ge \varepsilon_0$ for $x \in [1,\infty)$. Let $c > 0$ be any constant. 
For example, we shall set 
\begin{align*}
 f_2 (r) = \exp \left( c \int_1^r \frac{\Phi_3 (t)}{t} dt \right) ~;~
 h_2 (r) = \exp \left( c \int_1^r \frac{\Phi_4 (t)}{t} dt \right) ,
\end{align*}
and define a metric $g_2 = g_2(c)$ on $E=[1, \infty) \times S^1(1) \times S^1(1)$ by
\begin{align*}
  g_2 = g_2(c) := dr^2 + f_2(r)^2 g_{S^1(1)} + h_2(r)^2 g_{S^1(1)} .
\end{align*}
Then, $\displaystyle \Delta_{g_2} r \equiv \frac{c}{r}$ and $\displaystyle \nabla dr \ge \frac{c \varepsilon_0}{r} \,\widetilde{g}_2$ on $(E, g_2(c))$. 
\vspace{3mm}

Let $M_0^3$ be any $3$-dimansional compact $C^{\infty}$-manifolds with boundary $\partial M_0^3$. 
Assume that $\partial M_0^3$ consists of a disjoint union of finitely many $T^2$.
For example, we shall take $M_0^3 = [-1, 1] \times T^2$. 
\vspace{2mm}

\noindent (a) Firstly, we shall attach $(E, g_1(c_1))$ and $(E, g_1(c_2))$ to boundaries $\{ -1 \} \times T^2$ and $\{ 1 \} \times T^2$ of $[-1, 1] \times T^2$, respectively; after that, we shall extends the metrics $g_1(c_1)$ and $g_1(c_2)$, to a metric $g$ on $M^3 := \mathbb{R} \times T^2$. 
Assume that $c_1 < c_2$. 
Then, the limiting absorption principle holds on respectively $\{ x + iy \in \mathbb{C} \mid x > \frac{(c_1)^2}{4},\,x \neq \frac{(c_2)^2}{4},\,y \ge 0 \}$ and $\{ x + iy \in \mathbb{C} \mid x > \frac{(c_1)^2}{4},\,x \neq \frac{(c_2)^2}{4},\,y \le 0 \}$, and $-\Delta_g$ on $L^2(M^3, v_g)$ is absolutely continuous on $( \frac{\min\{ c_1, c_2 \}^2}{4}, \infty )$ by Theorem $1.1$ and Theorem $1.2$ in \cite{K4}. 
\vspace{2mm}

\noindent (b) Secondly, we shall attach $(E, g_1(c_1))$ and $(E, g_2(c_2))$ to boundaries $\{ -1 \} \times T^2$ and $\{ 1 \} \times T^2$ of $[-1, 1] \times T^2$, respectively; after that, we shall extends the metrics, $g_1(c_1)$ and $g_2(c_2)$, to a metric $g$ on $M^3 := \mathbb{R} \times T^2$. 
Then, the limiting absorption principle holds on respectively $\{ x + iy \mid x > 0,\,x \neq \frac{(c_1)^2}{4},\,y \ge 0 \}$ and $\{ x + iy \mid x > 0,\,x \neq \frac{(c_1)^2}{4},\,y \le 0 \}$, and $-\Delta_g$ on $L^2(M^3, v_g)$ is absolutely continuous on $( 0, \infty )$ by Theorem $1.1$ and Theorem $1.2$ in \cite{K4}. 
Note that, as for the merely absence of eigenvalues, ``small perturbation $\frac{\varepsilon}{r}$'' of $\Delta_g r$ is allowed on an end $[1, \infty) \times T^2$; see Corollary $1.3$. 
\vspace{2mm}

\noindent (c) Thirdly, we shall attach $(E, g_2(c_1))$ and $(E, g_2(c_2))$ to boundaries $\{ -1 \} \times T^2$ and $\{ 1 \} \times T^2$ of $M_0^3 = [-1, 1] \times T^2$, respectively; after that, we shall extends the metrics, $g_2(c_1)$ and $g_2(c_2)$, to a metric $g$ on $M^3:= \mathbb{R} \times T^2$. 
Then, the limiting absorption principle holds on respectively $\{ x + iy \mid x > 0,\,y \ge 0 \}$ and $\{ x + iy \mid x > 0,\,y \le 0 \}$, and $-\Delta_g$ on $L^2(M^3, v_g)$ is absolutely continuous on $( 0, \infty )$ by Theorem $1.1$ and Theorem $1.2$ in \cite{K4}. 
Note that, as for the merely absence of eigenvalues, ``small perturbation $\frac{\varepsilon}{r}$'' of $\Delta_g r$ is allowed on both ends; see Corollary $1.3$. 
\vspace{2mm}

To see the complexity of the growth order of $g_1$ and $g_2$ near the infinity, we shall consider the following example. 
Let $0 < \varepsilon_0 \ll 1$ be a small constant. 
Let $\{ a_n \}_{n =1}^{\infty}$ be any increasing sequence satisfying 
\begin{align}
  a_1=1~; \quad a_i < a_j\quad {\rm for~any}~i<j~; \quad \lim_{n\to \infty} a_n = \infty; 
\end{align}
let $\{ b_n \}_{n =1}^{\infty}$ be any sequence of positive numbers satisfying
\begin{align}
  \frac{\varepsilon_0}{a_{2n-2}} \le \min \{ b_n , 1 - b_n \} 
  \quad {\rm for}~n \ge 2 .
\end{align}
We shall take a $C^{\infty}$-function $\Phi_1: [1, \infty) \to (0,1)$ so that 
\begin{align*}
& \Phi_1(x) = b_n \quad {\rm for}~x \in [a_{2n-1}, a_{2n}]~{\rm and}~n \ge 1;~\\
& \Phi_1~{\rm is~monotone~on}~[a_{2n}, a_{2n+1}] \quad {\rm for}~n \ge 1.
\end{align*}
Then, $\{ \Phi_1, \Phi_2:=1 - \Phi \}$ is a partition of unity on $[1, \infty)$ satisfying $\Phi_1(x), \Phi_2(x) \ge \frac{\varepsilon_0}{x}$. 
However, since the choices of $\{ a_n \}$ and $\{ b_n \}$ satisfying $(49)$ and $(50)$ is arbitral, growth orders of metrics $g_1$ and $g_2$ can be very complicated near the infty. 
For example, consider the case of a random choice of $\{ b_n \}$ and $\lim_{n \to \infty} (a_{n+1} - a_n)=0$. 
\vspace{3mm}

%%
%%%%%     4-dimendional case     %%%%%
%%
\noindent {\bf Punctured compact $4$-manifolds.} 
In $4$-dimendional case, by removing finite points form {\it any} compact manifold $M_0^4$ without boundary, we can obtain many examples, because $S^3(1) = SU(2)$ is a Lie group. 
Let $X_1$, $X_2$, and $X_3$ is a left invariant orthonormal frame on $SU(2)$ with respect to $-B$ such that 
\begin{align*}
  [X_1, X_2] = 2 X_3,~~[X_2, X_3] = 2 X_1,~~[X_3, X_1] = 2 X_2, 
\end{align*}
where $B$ stands for the Killing form on $\mathfrak{su}(2)$; let $\omega_1$, $\omega_2$, and $\omega_3$ be left invariant $1$-forms dual to $X_1$, $X_2$, and $X_3$, respectively. 
\vspace{2mm}

\noindent (v)~Let $0 < \varepsilon_0 \ll 1$ be any small constant, and $\{ \Phi_1, \Phi_2, \Phi_3 \}$ be any $C^{\infty}$-partition of unity on $[1, \infty)$ satisfying $\displaystyle\Phi_j (x) \ge \frac{\varepsilon_0}{x}$ for $x \in [1, \infty)$ and $j = 1, 2, 3$; let $\beta > 0$ be any constant, and set 
\begin{align*}
  \phi_j(r) := \exp \left( \beta \int_1^r \Phi_j(t) \, dt \right) \qquad
  {\rm for}~~j=1, 2, 3. 
\end{align*}
We shall define a Riemannian metric $g_3$ on $E := [1, \infty) \times S^3(1)$ by
\begin{align*}
  g_3 = g_3 (\beta) := 
  dr^2 + \phi_1(r)^2 \omega_1 + \phi_2(r)^2 \omega_2 + \phi_3(r)^2 \omega_3. 
\end{align*}
Then, $\Delta_{g_3} r \equiv \beta$ and $\displaystyle\nabla dr \ge \frac{\beta \varepsilon_0}{r} \widetilde{g}_3$ on $(E, g_3(\beta))$. 
\vspace{2mm}

\noindent (vi)~Let $0 < \varepsilon_0 \ll 1$ be any small constant, and $\{ \Phi_1, \Phi_2, \Phi_3 \}$ be any $C^{\infty}$-partition of unity on $[1, \infty)$ satisfying $\Phi_j (x) \ge \varepsilon_0$ for $x \in [1, \infty)$ and $j = 1, 2, 3$; let $\beta > 0$ be any constant, and set 
\begin{align*}
  \phi_j(r) := \exp \left( \beta \int_1^r \frac{\Phi_j(t)}{t} \, dt \right) \qquad
  {\rm for}~~j=1, 2, 3. 
\end{align*}
We shall define a Riemannian metric $g_4$ on $E := [1, \infty) \times S^3(1)$ by
\begin{align*}
  g_4 = g_4 (\beta) := 
  dr^2 + \phi_1(r)^2 \omega_1 + \phi_2(r)^2 \omega_2 + \phi_3(r)^2 \omega_3. 
\end{align*}
Then, $\displaystyle\Delta_{g_4} r \equiv \frac{\beta}{r}$ and $\displaystyle\nabla dr \ge \frac{\beta \varepsilon_0}{r} \widetilde{g}_4$ on $(E, g_4(\beta))$. 
\vspace{2mm}

Let $M_0^4$ be any compact $4$-dimensional $C^{\infty}$-manifold without boundary and $p_1, p_2, \cdots, p_m$ be a points of $M_0^4$. 
Let $(E_1, g_4(\beta_1)), \cdots, (E_{m_0}, g_4(\beta_{m_0}))$ and \linebreak $(E_{m_0+1}, g_3(\beta_{m_0+1})), \cdots, (E_{m}, g_3(\beta_{m}))$ be ends as is stated above, where $E_j = E = [1, \infty) \times S^3(1)$ for $1 \le j \le m$ and $0 \le m_0 < m$ are integers. 
If $m_0=0$, we shall mean that there is no end satisfying (iv) for any constant $\beta > 0$. 
We shall expand a neighborhood around the point $p_j$, attach end $E_j$ stated above for $1 \le j \le m$, and define a metric $g$ on $M^4 := M_0^4 \, \sharp \, E_1 \, \sharp \cdots \, \sharp \, E_m = M_0^4 \, \sharp \, m E$ so that $g|_{E_j} = g_4(\beta_j)$ for $1 \le j \le m_0$; $g|_{E_j} = g_3(\beta_j)$ for $m_0 + 1 \le j \le m$. 
Then, the limiting absorption principle holds on respectively $\{ x + iy \in \mathbb{C} \mid x > 0,\, x \neq \frac{(\beta_j)^2}{4},\,j = m_0 + 1, \cdots, m,\,y \ge 0 \}$ and $\{ x + iy \in \mathbb{C} \mid x > 0,\, x \neq \frac{(\beta_j)^2}{4},\,j = m_0 + 1, \cdots, m,\,y \le 0 \}$, and $-\Delta_g$ on $L^2(M^4, v_g)$ is absolutely continuous on $( \frac{(\beta_{\min})^2}{4}, \infty )$ by Theorem $1.1$ and Theorem $1.2$ in \cite{K4}, where $\beta_{\min} := \{ \beta_j \mid j=1, \cdots, m\}$. 
Note that, as for the merely absence of eigenvalues, ``small perturbation $\frac{\varepsilon}{r}$'' of $\Delta_g r$ is allowed on ends $(E_j, g_4(\beta_j))$ for $1 \le j \le m_0$, if $m_0 \ge 1$; see Corollary $1.3$. 
\vspace{2mm}

Growth orders of metrics $g_3$ and $g_4$ on $[1,\infty) \times SU(2)$ near the infinity can be more complicated than those of the case of $[1,\infty) \times T^2$, because freedom of choice of partition of unity increases: $\sharp\{ \Phi_1, \Phi_2, \Phi_3 \} > \sharp\{ \Phi_1, \Phi_2 \}$.

%%
%%%%%     Reference     %%%%%
%%

%%
%%
\vspace{7mm}

\end{document}